\documentclass[11pt]{article}
\usepackage{setspace}
\usepackage{amsthm,amsfonts,amsmath}
\usepackage{amssymb, amsthm, algorithm, algorithmic,thmtools,thm-restate,enumerate,verbatim,bm,color}
\definecolor{darkgreen}{rgb}{0.1,0.4,0.1}

\usepackage{graphicx}
\usepackage{multirow}

\usepackage{enumerate}
\usepackage[authoryear,round]{natbib}

\usepackage{amssymb}
\usepackage{amsthm}
\usepackage{mathrsfs}
\usepackage{algorithm}
\usepackage{algorithmic}

\usepackage{setspace}

\usepackage{fancybox}

\newenvironment{proofof}[1]{{{\noindent \em {Proof of}} #1.}}{\hfill\rule{2mm}{2mm}}

\DeclareMathOperator*{\argmax}{arg\,max}

\usepackage{fullpage}

\renewcommand{\labelenumi}{(\roman{enumi})}

\newcommand{\be}{\begin{equation}}
\newcommand{\ee}{\end{equation}}
\newcommand{\bew}{\begin{equation*}}
\newcommand{\eew}{\end{equation*}}

\newtheorem{thm}{Theorem}[section]
\newtheorem{theorem}{Theorem}[section]

\newtheorem{definition}[thm]{Definition}

\newtheorem{proposition}[thm]{Proposition}

\newtheorem{lemma}[thm]{Lemma}
\newtheorem{remark}[thm]{Remark}

\newtheorem{assumption}[thm]{Assumption}

\numberwithin{equation}{section}

\begin{document}
\allowdisplaybreaks
\renewcommand{\labelenumi}{(\arabic{enumi})}

\title{Exact Simulation of Non-stationary \\ Reflected Brownian Motion}

\author{Mohammad Mousavi\footnote{Corresponding author. Department of Management Science \& Engineering,
Stanford University, email: {\tt mousavi@stanford.edu},
web: {\tt www.stanford.edu/$\sim$mousavi}.},\; Peter W. Glynn\footnote{Department of Management Science \& Engineering,
Stanford University, email: {\tt glynn@stanford.edu},
web: {\tt www.stanford.edu/$\sim$glynn}.}} \normalsize

\date{\today} 
\maketitle

\begin{abstract}
This paper develops the first method for the exact simulation of reflected Brownian motion (RBM) with non-stationary drift and infinitesimal variance. The running time of generating exact samples of non-stationary RBM at any time $t$ is uniformly bounded by $\mathcal{O}(1/\bar\gamma^2)$ where $\bar\gamma$ is the average drift of the process. The method can be used as a guide for planning simulations of complex queueing systems with non-stationary arrival rates and/or service time.
\end{abstract}

\newpage

\newcommand{\Q}{\mathbb{Q}}
\renewcommand{\P}{\mathbb{P}}
\newcommand{\norm}[1]{\|#1\|}
\newcommand{\nice }{\emph{nice }}
\newcommand{\nf }{as $n\rightarrow \infty$}
\newcommand{\ds }{ \buildrel D\over= }
\newcommand{\eqdist }{\ds}
\newcommand{\df }{\triangleq}
\newcommand{\EE}{\mathbb{E}}
\def\var{\mathop{\rm var}\nolimits} 

\newcommand{\ALGExact}{2}
\newcommand{\ALGAR}{1}

\setcounter{section}{0}
\setcounter{section}{0}

\section{Introduction}

This paper is concerned with the exact simulation of reflected Brownian
motion (RBM) with non-stationary drift and infinitesimal variance. Our
interest in this model stems from the fact that RBM is commonly used as a
stylized representation of a single-station queue (and often as a model
for extracting numerical approximations to queues in heavy traffic);
see \citet{heavy_traffic}.

In many (indeed most) real-world applications of queueing models, there exist
non-stationarities in the arrival rates and/or service time requirements
that are induced by time-of-day, day-of-week, or seasonality effects. In
addition, in some situations (as in production or inventory contexts),
there may also be non-stationarities associated with rising or falling
demand for a product, as it is introduced to the marketplace or becomes
obsolete. In such applications, a simplified description of the workload
process $X=(X(t):t\geq 0)$ is to postulate that it satisfies the stochastic
differential equation (SDE)
\begin{align}\label{RBM}
dX(t)=\mu(t) dt+\sigma(t) dB(t) + dL(t),
\end{align}
where 
$L=(L(t):t\geq 0)$ is the continuous non-decreasing process
satisfying $$I(X(t)>0) dL(t)=0$$ for $t\geq 0$, $B= (B(t)\colon t\geq
0)$ is standard Brownian motion, and  $\mu=(\mu(t)\colon t\geq 0)$ and
$\sigma=(\sigma(t)\colon t\geq 0)$ are given (measurable) deterministic
functions. Note that this stylized model permits both the instantaneous
drift and volatility to be separately specified, unlike the non-stationary
$M(t)/M(t)/1$ model that has been previously studied in the queueing
literature (see, for example, \citet{Massey}) in which the instantaneous drift
must always match the instantaneous volatility. Assuming $X(0)=x$, our goal
here is to provide an algorithm for generating $X(t)$  with a complexity that
is bounded in $t$, at least when $X$ empties infinitely often almost 
surely
(a.s.). If the coefficient functions are stationary (so that $\mu(\cdot)$
and $\sigma(\cdot)$ are constant) and we send $t\rightarrow \infty$, it is
evident that this is a non-stationary analog to the \textit{exact simulation}
problem for positive recurrent Markov processes. Hence, we use the terminology
``exact simulation'' to also refer to our non-stationary problem.

Of course, if RBM has stationary drift and infinitesimal variance, the
transient and steady-state distributions of $X$ are then known in closed
form, and simulation is unnecessary.
In the non-stationary context, the transition density $p(t,x,y) dy\
(\stackrel{\Delta}{=}P(X(t)\in dy | X(0)=x))$ would be expected to satisfy
the Kolmogorov forward partial differential equation (PDE)
\begin{align}
\frac{\partial}{\partial s} p(s,x,y)=\frac{1}{2} \sigma^2(s)
\frac{\partial^2}{\partial y^2} p(s,x,y)- \mu(s) \frac{\partial}{\partial y}
p(s,x,y), \mbox{\  \  \  \  \  \ $0< s\leq t,$}
\end{align}
subject to $P(X(0)\in dy | X(0)=x)=\delta_x(dy)$ and
\[
 \frac{\sigma(s)^2}{2} \frac{\partial}{\partial y} p(s,x,0)- \mu(s)
 p(s,x,0)=0,  \mbox{\  \  \  \  \  \ $0< s\leq t.$}
\]
 Unlike the stationary case, this PDE has no known closed-form solution,
 and would need to be solved numerically. This paper provides an efficient
 computational alternative to
 numerically solving the above PDE, which is especially attractive when
 the time horizon $t$ of interest is large.

As indicated above, $X$ can be used as a basis model for studying a queue
with non-stationary dynamics. But in many applications, we would prefer
to use a ``finer grain" and more realistic simulation model, rather than
RMB itself, as a mathematical description of the real-world system under
consideration. One intuitively expects that such models ``lose memory",
in the sense that the distribution at time $t$ often will be insensitive
to the state at time $t-u$, provided that $u$ is chosen large enough. In
the presence of such insensitivity, one can (for example) initialize
the system in the empty state at time $t-u$ and execute the fine-grain
simulation only over $[t-u,t]$ rather than $[0,t]$, thereby generating
significant computational savings. Thus, identifying an appropriate value
of $u$ is of significant interest.

While estimating the ``loss of memory" for the underlying detailed model would
be extremely challenging, we will argue in Section 2, via a coupling argument,
that it can be readily estimated for the simplified RBM model. Simulation of
the non-stationary RBM can then be used to determine how large $u$ should
be chosen for the detailed model (perhaps by multiplying the RBM's value
of $u$ by a factor of 2, in order to account for the model approximation
error). Thus, non-stationary RBM can be viewed as a simulation planning
tool for more complex detailed queueing simulations, in the same sense that
\citet{whitt1989planning} and \citet{asmussen1992queueing} argue that RBM with
stationary dynamics is an appropriate tool for planning steady-state
queueing simulations.

 We note, in passing, that a special case of our problem arises
 when $\mu(\cdot)$ and $\sigma(\cdot)$ are periodic (with the same
 period). Various theoretical results are known for such periodic models;
 see for example, \citet{Lemoine}, \citet{ PWhitt}, \citet{Thorisson},
 and \citet{Bambos}. In addition, when  $\mu(\cdot)$ and $\sigma(\cdot)$
 are stochastic (but independent of $B=(B(t):t\geq 0))$, the problem of
 simulating $X$ reduces to the deterministic case considered here upon
 conditioning on $\mu$ and $\sigma$. In this way, our method can cover
 (for example) RBM with non-Markov $\mu$ and $\sigma$.

The remainder of this paper is organized as follows: Section 2 discusses
how to plan the simulation of a queueing model with non-stationary
inputs. Sections 3 and 4 
 develop the main exact simulation methods of
time-dependent RBM. 
Sections 5 and 6 explain the implementation details. Section
6 analyzes the algorithm. Section 7 provides numerical results.

\section[Planning Simulations of Non-stationary Queueing Models]{Planning Simulations of Non-stationary\\Queueing Models}

Our 
goal here is to study the rate at which the non-stationary RBM 
$X$
``loses memory''. In view of our discussion in the 
introduction, we wish
to specifically answer the following question:\\
\begin{center}
\parbox{12cm}{
\textit{How large must we choose $u$ so that the RBM started at
time $t-u$ from the idle state (i.e. no workload in the system) will have
a distribution at time $t$ that is close to that of the RBM at time $0$
from its initial workload $x$?\\}
}\end{center}
For $0\leq s \leq t$, let $(X_s(t): t\geq s)$ be the solution to (\ref{RBM})
conditional on $X_s(s) = 0$, and let $\| \cdot \|$ be the total variation
norm. Also, set
$$\hat {v}_t = \sup{\{r \in [0,t]\colon X(r)=0}\}$$
with the convention that if $X$ does not visit $0$ over $[0,t]$, then
$\tilde{v}_t=-\infty$.

\begin{proposition} \label{coupling}
For $0 \leq s \leq t$,
\begin{align*}
\| \P(X_s(t) \in \cdot)-\P(X(t) \in \cdot \mid X(0)=x) \|  \leq \P(\tilde{v}_t
\leq s \mid X(0)=x).
\end{align*}
\end{proposition}

Proposition \ref{coupling} provides an answer to our question: For a given
error tolerance $\epsilon$, we should choose $u$ so that
\begin{align*}
\P(\tilde{v}_t \leq t-u \mid X(0)=x) = \P(t-\tilde{v}_t \geq u \mid X(0)=x)
\leq \epsilon .
\end{align*}
\\
\begin{proofof} {\textit{Proposition \ref{coupling}}}
Let
\begin{align*}
\widetilde{Y}(t)&= \int_0^t \! \mu(s) \, ds +  \int_0^t \! \sigma(s) \,
dB(s) \\
\widetilde{M}(s,t) &= -\displaystyle\inf_{s \leq r \leq t} \widetilde{Y}(r)
\end{align*}
for 
$0\leq s \leq t$. It is well known that the 
solution to the SDE
(\ref{RBM}), conditional on $X(0)=x$, can be explicitly compared in terms
of $\widetilde{Y}$; 
specifically,
\begin{align}
X(t)=x+ \widetilde{Y}(t)+ \displaystyle\sup_{0 \leq s \leq
t}{\max\left(-x-\widetilde{Y}(s), 0\right)};\label{x_exp}
\end{align}
see, for example, p. 20 of \citet{harrison}.

Relation (\ref{x_exp}) can be re-expressed as
\begin{align} \label{eq.e2}
X(t) = \begin{cases} \widetilde{Y}(t)+\widetilde{M}(0,t) &\mbox{if }
\widetilde{M}(0,t) \geq x, \\
\widetilde{Y}(t)+x &\mbox{if } \tilde{M}(0,t) \leq x. \end{cases}
\end{align}
We can couple $X_s$ to $X$ by noting  that $X_s(t)$ can similarly be
expressed as
\begin{align*}
X_s(t) = \widetilde{Y}(t)+\widetilde{M}(s,t).
\end{align*}
Hence, $X_s(t)=X(t)$ whenever $\widetilde{M}(s,t) = \widetilde{M}(0,t) \geq
x$. But $\widetilde{M}(s,t)=\widetilde{M}(0,t) \geq x$ precisely whenever 
$X$
visits $0$ on $[s,t]$, which is equivalent to assertions that $\widetilde{v}_t
\geq s$. Consequently, 
coupling between $X$ and $X_s$ establishes that
\begin{align*}
\P(X_s(t) \in \cdot\ ,\ \tilde{v}_t \geq s) = \P( X(t) \in \cdot\ ,
\tilde{v}_t \geq s \mid X(0) = x),
\end{align*}
proving that
\begin{align*}
\displaystyle\sup_{B} \Big \vert \P(X_s(t) \in B) - \P(X(t) \in B \mid
X(0)=x )\Big \vert \leq \P(\tilde{v}_t < s \mid X(0)=x).
\end{align*}

\end{proofof}

Thus, the key random variable (rv) to simulate for purposes of planning
queueing simulations is $t-\tilde{v}_t$. On the other hand, when $X$ itself
is the best  model of interests, our focus is on $X(t)$. So, our goal in
this paper is to efficiently simulate the pair $(t-\tilde{v}_t, X(t))$. In
particular, our interest is in an algorithm for generating $(t-\tilde{v}_t,
X(t))$ having a computational complexity independent of 
$t$.

To pursue this objective, it is convenient to use a distributionally
equivalent representation for $(t-\tilde{v}_t, X(t))$. To this end, let
\begin{align*}
Y^\star(r) = \widetilde{Y}(t) - \widetilde{Y}(t-r)
\end{align*}
be the time-reversal of $\widetilde{Y}$ (time-reversed from time $t$). We now
express $t-\tilde{v}_t$ and $X(t)$ in terms of the time-reversal $Y^\star$.

Note that
\begin{align} \label{eq.e0}
X(t) = \begin{cases} \displaystyle\sup_{0\leq s\leq t} Y^\star(s),  &\mbox{if
} \displaystyle\sup_{0\leq s\leq t} Y^\star(s) \geq x+Y^\star(t) \\
\\
x+Y^\star(t), & \mbox{if} \displaystyle\sup_{0\leq s\leq t} Y^\star(s) <
x+Y^\star(t) . \end{cases}
\end{align}

As for $\tilde{v}_t$, it is (in 
the event $\{\widetilde{M}(0,t)\geq x\}$)
the largest value of $r\in [0,t]$  for which
\begin{align*}
\widetilde{Y}(r) = \inf_{0\leq s\leq r} \widetilde{Y}(s).
\end{align*}
Since $\tilde{v}_t$ is the largest such value and $\widetilde{Y}$ is
continuous,
\begin{align*}
\widetilde{Y}(v) > \inf_{0\leq s\leq t} \widetilde{Y}(s).
\end{align*}
for $v\in [\tilde{v}_t,t]$. Consequently,
\begin{align*}
\widetilde{Y}(\tilde{v}_t) = \inf_{0\leq s\leq t}\widetilde{Y}(s).
\end{align*}
Because $\widetilde{Y}(\cdot)$ has a unique minimizer on $[0,t]$,
$\tilde{v}_t=\arg\min (\widetilde{Y}(s)\colon 0\leq s\leq t)$. It is then
immediate that $v^\star_t=t-\tilde{v}_t$, where
\begin{align*}
v^\star_t = \arg\max \Big(Y^{\star}(s): 0\leq s\leq t\Big).
\end{align*}
We conclude that
\begin{align} \label{eq.e1}
t-\tilde{v}_t = \begin{cases} v^\star_t &\mbox{if }\displaystyle \sup_{0\leq
s\leq t} Y^\star(s) \geq x+ Y^\star (t) \\
\infty & \mbox{else } . \end{cases}
\end{align}
Thus, the random vector $(t-\tilde{v}_t, X(t))$ is determined by
$\Big(v^\star_t,\displaystyle\sup_{0\leq s\leq t} Y^\star (s),
Y^\star(t)\Big)$.

Finally, note that $(Y^\star (s)\colon s\geq 0)$ has the same law as
the process
\begin{align*}
Y (s) = \int_{0}^s \mu' (r) dr + \int_{0}^s \sigma' (r) d B(r),
\end{align*}
where $\mu'(r)\df \mu(t-r)$ and $\sigma'(r)=\sigma(t-r)$ for $r\geq0$. Thus,
the remainder of this section is concerned with generating the triplet
$(v_t,M(t),Y(t))$ ($ \ds(v^\star_t, \displaystyle\sup_{0\leq s\leq t}
Y^\star (s), Y^\star (t))$), 
where
\begin{align}
&M(t)=\displaystyle\max_{0\leq s\leq t} Y(s)\label{max_m} \\
&v_t = \arg\max (Y(s): 0\leq s \leq t),\nonumber
\end{align}
and $ \buildrel D\over=$ denotes equality in distribution.

Suppose now that the drift and volatility functions are periodic. Without loss
of generality, we assume that the period is one. In this periodic setting,
$\mu'(\cdot)=\mu(n+b-\cdot)$ and $\sigma'(\cdot)=\sigma(n+b-\cdot)$
are independent of $n\in \mathbb{Z}_+$ for $0\leq b\leq 1$, so that
our above argument establishes that $(n+b-\tilde{v}_{n+b},X(n+b))$ can
be determined from the joint distribution of $\left(v_{n+b}, M(n+ b),
Y(n+b)\right)$. Furthermore, if
\begin{align}
\int_{0}^1 \mu (s) \operatorname{d} s < 0,\label{pcondition}
\end{align}
then $Y(t)\rightarrow -\infty $ a.s. as $t\rightarrow \infty$, 
and there
exist finite-valued rev's $v^b_\infty$ and $M^b_\infty$ such that
\begin{align*}
(v_t, M(t)) \rightarrow (v^b_\infty, M^b_\infty)
\end{align*}
as $t\rightarrow \infty$. In addition,
\begin{align*}
X(n+b) \Rightarrow M_\infty^b
\end{align*}
as $n\rightarrow \infty$. Hence, if $(v_t, M(t), Y(t))$ can be generated
with a complexity independent of $t$, one arrive at an algorithm that can
generate $(n+b-\tilde{v}_{n+b}, X(n+b))$ with a complexity independent of
$n$. Furthermore, if $M_\infty^b$ can be sampled, this converts to an exact
sampling algorithm for the equilibrium of $(X(n+b): n\geq 0)$.

This leads 
naturally to the following re-formulation of the key issue to be
addressed in this 
chapter:\\
\begin{center}
\parbox{12cm}{
\textit{Under appropriate hypotheses, provide an algorithm for generating
$M(\infty)$ in finite expected time, and generating $(v_t, M(t), Y(t))$
with a complexity independent of $t$.\\
}}
\end{center}

\noindent For a non-periodic specification of $\mu'(\cdot)$ and
$\sigma'(\cdot)$ for which $Y(t)\rightarrow -\infty$ as $t\rightarrow
\infty$ a.s., we can view a solution to the above problem as providing an
exact sampling algorithm for the rv $X(0)$, assuming $X$ was initialized at
time $-\infty$ and evolves 
according to the drift $\mu'(-r)$ and volatility
$\sigma'(-r)$ for $r\in (-\infty, 0]$. We note that a sufficient condition
for $Y(t)\rightarrow \infty$ a.s. is that $\int_0^t\mu'(s)ds \rightarrow
-\infty$,  $\int_0^t\sigma'^2(s)ds \rightarrow \infty$, and
\begin{align*}
			\frac{\sqrt{\int_{0}^{t} \! \sigma^{\prime 2}(s)ds
			\cdot \log\log (\int_{0}^{t} \! \sigma^{\prime 2}(s)
			\, d s)}}{\int_{0}^{t} \! \mu^\prime(s) \, d s}
			\rightarrow 0
		\end{align*}
		as $t \rightarrow \infty$. This follows from the law of
		the iterated logarithm for $B$ 
and the fact that
		\begin{align}\label{TC}
			\left( \int_{0}^{t} \sigma^\prime d B(s)\colon
			t\geq 0\right) \ds\left(B\left( \int_{0}^{t}
			\! \sigma^{\prime 2}(s) \, d s\right)\colon t\geq
			0\right).
		\end{align}
To simplify the notation over the reminder of this paper, we henceforth
denote $\mu'$ and $\sigma'$ by $\mu$ and $\sigma$, respectively.

\section{Our First Proposed Algorithm}
	As argued in Section 2, our interest in this paper is to efficiently
	simulate the triplet $(v_t,M(t),Y(t))$, where $Y$ is given by
	\begin{align}
		Y(t) &= \int_{0}^{t} \! \mu (s) \, d s + \int_{0}^{t}
		\! \sigma (s) \, d B(s).
	\end{align}
	We assume that:
	\begin{assumption}\label{A1}
	 The function $\mu (s)$ and $\sigma^2(s)$ are differentiable and
	 $\sigma^2(t)>0$ for $t\geq 0$.
	\end{assumption}
	Set
	\begin{align*}
		\Lambda (t) &= \int_{0}^{t} \! \sigma^2(s) \, d s
	\end{align*}
	and let $\Lambda^{-1}(\cdot)$ be its function inverse. In view
	of (\ref{TC}), it is immediate that $Y$ has the same law as
	$(Z(\Lambda(t))\colon t\geq 0)$, 
where
	\begin{align}
		Z(t) &= \int_{0}^{t} \! \gamma(s) \, d s+B(t)
	\end{align}
	and
	\begin{align*}
		\gamma(t) \df
		\frac{\mu(\Lambda^{-1}(t))}{\sigma^2(\Lambda^{-1}(t))}.
	\end{align*}
	Set
        $$
\displaylines{
		m(t) \df \max_{0\leq s \leq t}Z(s)  \cr
		\eta_t \df \argmax(Z(s):0\leq s \leq t).}
                $$
	Because
	\begin{align}
		(v_t, M(t),Y(t)) &\ds
		(\Lambda^{-1}(\eta_{\Lambda(t)}),m(\Lambda(t)),Z(\Lambda(t)))\label{time_T}
	\end{align}
	for $t \geq 0$ 
and
	\begin{align*}
		(v_\infty , M(\infty)) \ds (\Lambda ^{-1}(\eta_\infty),
		m(\infty)),
	\end{align*}
	our problem therefore reduces to the study of $(\eta_t,m(t),Z(t))$
	and $(\eta_\infty,m(\infty))$. We now further 
assume:
		\begin{assumption}\label{A2}
		There exist 
$d>0$ and $\bar\gamma > 0$ such that
		 \begin{align*} \int_{s}^{t} \! \gamma(u) \, d u \leq d
		 -(t-s)\cdot  \bar \gamma\end{align*} for $0 \leq s \leq t$.
	\end{assumption}
	Assumption \ref{A2} is clearly simplified in the periodic setting
	under (\ref{pcondition}).

	We start by noting that $\eta_t$ can be bounded under 
Assumption
	\ref{A2}. Observe that
	\begin{align}
		\eta_t &\leq \eta_\infty \nonumber\\
		&= \argmax(Z(s): s \geq 0) \nonumber\\
		&\leq \sup \{ s\geq 0: Z(s) \geq Z(0)\}\nonumber\\
		&= \sup\{ s \geq 0 : \int_{0}^{s} \! \gamma(u) \, d u +
		B(s) \geq 0 \} \label{3.3} \\
		&\leq \sup\{s\geq 0: B(s) \ge \bar \gamma s - d\} \nonumber \\
		&\df L. \nonumber
	\end{align}
	But
	\begin{align}
		L &= 1/\inf\{r\geq 0 : B(1/r) \geq \bar\gamma/r - d\}
		\nonumber \\
		&= 1/\inf \{r \geq 0 : rB(1/r) \geq \bar\gamma -dr \}
		\nonumber \\
		&\ds 1/\inf\{r \geq 0: B(r)+dr \geq \bar\gamma \}
		\label{3.4} \\
		&\df 1/H, \nonumber
	\end{align}
	since $(rB(1/r)\colon r\geq 0) \ds (B(r):r\geq 0)$; see, for example
	on \citet[p .104]{Shreve}. 
We conclude that
	\begin{align*}
		\P(\eta_t \geq u) &\leq \P(1/H \geq u) \\
		&= \P(H \leq 1/u).
	\end{align*}
	In view of $H$'s 
inverse Gaussian distribution (see p. 297 of
	\citet {Shreve}), 
we arrive at the following analytical bound on
	$\eta_t$'s probability.
	\begin{proposition}
	Under Assumptions \ref{A1} and \ref{A2},
	\begin{align*}
		\P(\eta_t \geq u) \leq \int_{0}^{1/u} \! \left(
		\frac{\bar\gamma^2}{2\pi x^3}\right)^{1/2}\exp \left(-
		\frac{(dx-\bar\gamma )^2}{2x}\right) \, d x.
	\end{align*}
	\end{proposition}
	This bound can be used directly to help plan the queueing simulations
	discussed earlier. In general, we expect this bound to be quite
	loose, so we will  achieve much better estimates of the approximation
	value of $u$ (for purposes of $X(t)$ by initializing the system in
	the empty state at time $t-u$) by simulating $\eta_t$ itself.

	However, the overestimations (\ref{3.3}) and (\ref{3.4}) are also
	very helpful in simulating $\eta_t$. Specifically, $\eta_t$ and
	$m(t)$ are determined by $(Z(s): 0\leq s\leq L)$ uniformly in $t$,
	so their generation of $(\eta(t),m(t))$ only involves simulating $Z$
	over a finite time horizon, independent of $t$.

	Given that the first step in such an algorithm involves generating
	$L,\ (Z(s)\colon 0\leq s \leq L)$ must be simulated conditional on
	$L$. Fortunately, much is learned about the distribution of $B$,
	conditional on $L$.
\paragraph{Result.}Conditional on $L,(B(s)\leq 0 \leq s \leq L)$ is
independent of $(B(s):s \geq L)$,
	\begin{align*}
		(B(s)\colon 0\leq s \leq L) &\ds \left( \sqrt{L} B^0(s/L)
		+ \bar\gamma s - \frac{sd}{L}:0\leq s \leq L\right),
	\end{align*}
	and
	\begin{align*}
		(B(L+s):s\geq 0) \ds (\beta(s): s\geq 0),
	\end{align*}
	where $(B^0(s): 0\leq s \leq 1)$ is a Brownian bridge 
process
	and $(\beta(s):s > 0)$ has a distribution given by
	\begin{align*}
		\P(\beta(\cdot) \in A) = \P(B(\cdot)\in A | B(t) \leq \bar
		\gamma t  \text{ for } t \geq 0).
	\end{align*}
	See p. 161 of \citet{bass2011stochastic} for details.

 These results lead to the following procedure for sampling $(\eta_{t},m(t))$:

  \begin{description}
     \item[Algorithm 1 : Exact Sampling of $(\eta_{t},m(t))$]
     \hspace{6cm}
      \begin{enumerate}[1]
      \item Generate a random variate $H$ from an inverse-Gaussian
      distribution with  mean $\mu =\bar \gamma/d$ and shape parameter	$
      \lambda =\bar \gamma^2$; set $L=1/H$ and $y=L\bar\gamma-d$.
      \item Letting $r=\min\{L,t\}$, draw a sample of the Brownian bridge
      $B(r)$ conditional on $B(L)=y$, and compute $Z(r)$.
      \item Conditional on $Z(r)$, draw a sample of $(\eta_{r},m(r))$;
      see Section \ref{SBB}.
\end{enumerate}
\end{description}
Note that for $t\geq L$, we have $m(t)=m(L)$ and $\eta_t=\eta_L$. Therefore,
the running time of this algorithm does not depend on $t$.
In the next section, we develop a more efficient algorithm in which we do
not require the simulation 
$Z$ over the entire interval $[0,L]$ to compute
the maximum of the process.

\section{Our Second Proposed Algorithm}

In this section, we provide an exact sampling method for generating the
triplet $(v_t, M(t), Y(t))$ whose complexity is independent of $t$. In
particular, this method can be used to generate $(v_\infty, M(\infty))$ in
finite expected time by assigning $t=\infty$. Aforementioned, by exploiting
the time transformation $\Lambda(\cdot)$ as per (\ref{time_T}), it is
sufficient to generate samples of $(\eta_{t},m(t),Z(t))$.

The key idea of the algorithm is based on constructing a finite
sequence of random disjoint intervals $(\alpha_1,\beta_1),\ldots,$
$(\alpha_K,\beta_K)$ such that the global maximum of $Z$ occurs over one of
these intervals. Suppose we can generate a sample path of $Z$ until $\beta_1$
such that $Z(\beta_1)$ is sufficiently lower than $m_1$, the maximum of the
sample path over $[0,\beta_1]$. In the 
case that the process $Z$ never rises up
the level $m_1$, the global maximum of $Z$ is $m_1$, and the procedure can
be terminated. Otherwise, one can generate a sample $\alpha_2$ from the
hitting time of the level $m_1$, and repeat the procedure starting from
$\alpha_2$. We show that in Lemma \ref{d_alpha} the procedure terminates
with at at least a constant probability in each iteration. Thus, the global
maximum is achieved in finite expected time. Figure \ref{algorithm1_figure}
illustrates this idea.
\begin{figure}[ht]\label{algorithm1_figure}
\begin{center}
  \caption{\small{Sample path of $Z(t)$ generated by Algorithm \ALGExact. }}
\includegraphics[ width=120 mm,height=80mm]{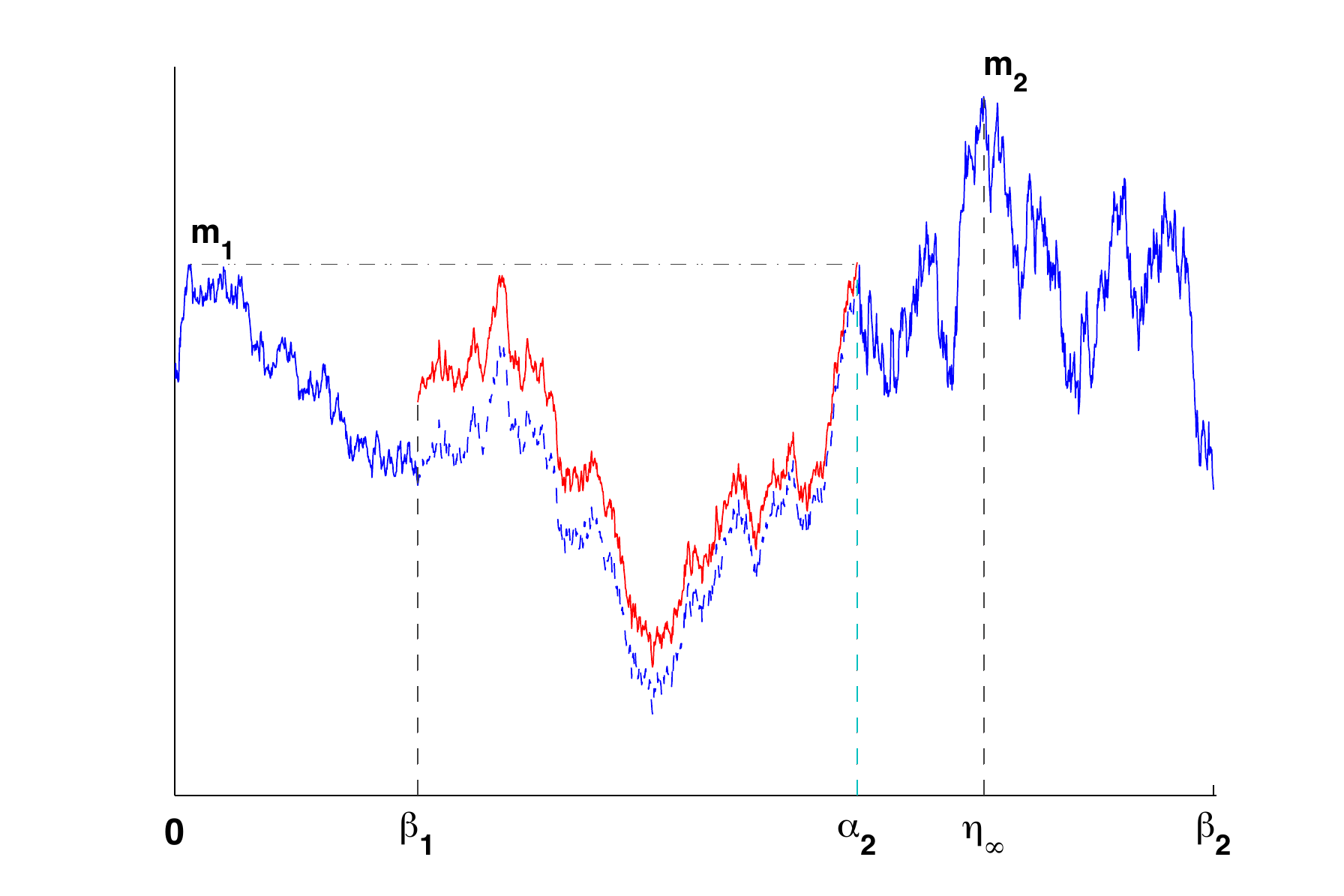}
\end{center}
{\small The red curve represents the dominating process $U_t$. The process
Z is not required to be simulated over the dashed-line. }
\end{figure}

To implement this idea, we need the following main 
components:
\begin{enumerate}[i.)]
\item A sampling mechanism to generate $(t_\zeta,\zeta,\beta)$. Here, $\beta$
is the first hitting time of the process $Z$ to a predetermined level,
$\zeta$ is the maximum of the process over the interval $[\alpha,\beta]$,
and $t_\zeta$ is the time at which the maximum occurs.
\item A testing procedure to check whether $Z(s)$ raises up level
$m(\beta)=\displaystyle\sup_{0\leq u\leq \beta} Z_u$ at some time $s\geq
\beta$.
\end{enumerate}
In Section \ref{TDBM}, we describe and analyze an algorithm to address
i.). For 
ii.), we construct a constant drift Brownian motion
$U^\beta=(U^\beta_s\colon s\geq \beta)$ dominating $Z(s)$. Let
\begin{align}\label{U}
U^\beta_s\df d+Z(\beta)+B(s)-B(\beta)- (s-\beta)\cdot \bar \gamma.
\end{align}
 By Assumption \ref{A2}, we can conclude that $Z(s)\leq U_{s}^\beta$ for
 all $s\geq \beta$. If for one of $\beta_k$ 
  $$U_{s}^{\beta_k} \leq m({\beta_k})=	\displaystyle\sup_{0\leq u\leq
  \beta} Z_u$$
holds true 
for all $s\geq \beta_k$, then $m({\beta_k})$ is the global maximum of
 $Z$. Since $U_{s}^\beta$ is a constant drift Brownian motion, we can
 precisely characterize its first hitting time distribution (see Lemma
 \ref{d_alpha}). Hence, we can easily check whether $U_{}^\beta$ hits
 level $m(\beta)$.

Below,
 we state the algorithm in 
detail. One can generate exact samples
of $(v_\infty, M(\infty))$ by setting $t=\infty$.
 \begin{description}
\item[Algorithm 2 : Exact Sampling of $(v_t, M(t), Y(t))$ ]\label{ALG1}
\hspace{6cm}
\begin{enumerate}[1]
\item Initialize $k=1$, $\zeta_0=\alpha_0=0$, $\eta_{t_0}=0$, and select
parameters $c>1$ 
and $\epsilon>0$.
   \item Generate a sample of $( t_\zeta, \zeta, \beta_{k})$ by subroutine
   \ref{ALG3}, where
 \begin{align}\label{beta2}
		 &  \beta_{k}=\inf \{\Lambda^{-1}(t)\geq s\geq \epsilon + \alpha_{k}
		 \colon Z(s)\leq m_{k}- cd\},\\
		  & \zeta_k=\displaystyle \max_{\alpha_{k}\leq s \leq
		  \beta_{k}}Z(s),\nonumber
\end{align}
and $t_{\zeta}$ is the time at which maximum occurs.
    \item \begin{description}
\item[-] If $\zeta_k>m_k$, 
update $m_{k+1}=\zeta$ and $\eta_{t_0}\leftarrow
t_{\zeta} $.
\item[-]  Otherwise, $m_{k+1} = m_{k} $.
\end{description}
  \item Sample $\alpha_{k+1}=\inf \{s\geq \beta_{k} \colon
  U_s^{\beta_k}\geq m_{k}\}$; see Lemma \ref{d_alpha}.
 \item If $\Lambda^{-1}(\alpha_{k+1})<t$, set $k\leftarrow k+1$, and repeat the procedure
 from Step 3.\\
 \hspace{4 in} Otherwise, terminate, and sample $Z(\Lambda(t))$ for $t<\infty$ conditional on $\alpha_{k+1}$.
  \item Return $(\Lambda^{-1}(\eta_{t_0}),m_k, Z(\Lambda(t)))$.

 \end{enumerate}

\end{description}
In Section \ref{analyseALG}, we show that
\[
\P(\alpha_k=\infty)\geq 1-\exp(2\bar\gamma(c-1)d).
\]
Therefore, the algorithm terminates in a 
finite number of iterations; call it
$K$. The expectation of $K$ has an upper bound independent of $t$ even in the
case $t=\infty$; see Theorem \ref{RT_alg1}. Furthermore, it is straightforward
to show that $\beta_K\leq L$; i.e, the length of the 
sample path generated by
Algorithm \ALGExact\ is less than $L$ which does not depend on $t$. Note that it is possible 
to determine whether $\Lambda^{-1}(\alpha_{k+1})<t$ or not in $\log(\alpha_{k+1})$, which is independent of $t$.

In the end, it is worth mentioning that we are not required 
to
simulate the process over the entire interval $[0,\beta_K]$ to compute
$\displaystyle\sup_{0\leq s\leq \beta_K}Z(s)$, since the maximum does not
occur 
between $\beta_k$ and $\alpha_{k+1}$ for $k=1,\ldots,K-1$. In contrast,
Algorithm 1, discussed in the previous section, requires generating 
the
full path of the process (conditional on $L$) over $[0,L]$. Clearly, this
is computationally more expensive than Algorithm  \ALGExact.

In the next section, we discuss Step 2, generating exact samples for the
maximum of $Z(t)$ over a finite interval.
\section[{Exact Sampling of Time-dependent Drift  Brownian Motion}]{Exact
Sampling of Time-dependent Drift  Brownian Motion}\label{TDBM}
This section describes an exact method for sampling the unit-volatility
time-dependent drift Brownian motion, and its maximum over a finite time
interval, 
which can be employed as a subroutine in Step 2 of Algorithm
\ALGExact. The method uses an acceptance/rejection mechanism similar to
that of \citet{Beskos2006}.

An acceptance/rejection scheme for exact simulation of state-dependent
diffusions was developed for certain one-dimensional
diffusions in \citet{Beskos2006}. The generation of the acceptance indicator
based on a thinning mechanism is proposed in \cite{beskos_retrospective}. Furthermore, this scheme is extended to a wider class of diffusions in \citet{chen2012}.  \citet{giesecke2011} generalized their method to jump-diffusions with state-dependent coefficients and jump intensity. Here, we develop a similar mechanism for time-dependent diffusions.

Recall that $Z =(Z(t) : t\geq 0)$ is a Brownian motion with time-dependent
drift $\gamma(t)$, so that the position at time $t$ is given by
\begin{align}
Z (t) = B(t) + \int_0 ^t  \gamma(s) ds.\label{gamma}
\end{align}

The objective of this section is to generate an exact sample of the maximum
of $Z $, before a fixed time $T_f$, or the first exiting time of the interval
$(v,u)$, 
where $v<0$ and $u>0$. More formally, we generate an exact sample of
$$(t_\zeta ,\sup_ {t \leq  T_f\wedge \tau_{v,u}}Z(t) ,	\tau_{v,u}),$$
where 
$$\tau_{v,u} = \inf \{t\geq 0: Z(t) \notin (v,u)\}$$
 is 
the first exit time of $Z$ from the interval $(v, u)$, and $t_\zeta$
 is the time at which the maximum of $Z$ over $[0,   T_f\wedge\tau_{v,u}]$
 occurs. We can assume that $T_f$ is equal to infinity, in that case
 generating exact samples of the first hitting time is possible.

The key idea is to generate a candidate sample path of a standard
Brownian motion and accept it as a sample path of the process $Z$ with
the probability proportional to the likelihood ratio between the law of
two processes. In Theorem \ref{likelihood}, we calculate this likelihood
ratio. Next, we construct a Bernoulli random variable $I$ for which $I=1$
with the acceptance probability proportional 
to the likelihood ratio, 
which
indicates the acceptance of the candidate.

 Let $\mathcal{F}=(\mathcal{F}_t: t\geq 0)$ be the filtration generated by
 the process $Z$ in (\ref{gamma}), and $\bar\tau$ be a stopping time with
 respect to this filtration. Let $ \P=	\P(z;s,t)$ be the probability measure
 on the $\sigma$-field $\mathcal{F}_{\bar \tau\wedge t}$ induced by the path
 $\{Z({u\wedge \bar \tau})\colon s \leq u \leq t \wedge \bar \tau\}$. Theorem
 \ref{likelihood} provides a formula for the likelihood ratio between $\P$
 and an equivalent measure $\Q = \Q(z;s,t)$ on $\mathcal{F}_{\bar \tau\wedge
 t}$ under which $\{Z({u\wedge \bar \tau})\colon s \leq u \leq t \wedge
 \bar \tau\}$ is a path of the standard Brownian motion stopped at $\bar
 \tau\wedge t$.

\begin{theorem}\label{likelihood}
Let $Z (t) = B(t) + \int_0 ^t  \gamma(s) ds$, $\gamma(t)$ be a continuously
differentiable function, and $\bar \tau$ be a finite value stopping time
with respect to the filtration $\mathcal{F}$. Then
for any event $B \in \mathcal{F}_{\bar \tau\wedge t}$, we have
\begin{align}
 \frac{\P(B)}{\Q(B)}=\EE^\Q\Big[\exp\left(\gamma(\bar \tau\wedge t)
 W^\Q_{\bar \tau\wedge t}-\frac{1}{2} \int_s^{\bar \tau\wedge t} \gamma^2(u)
 du - \int_s^{\bar \tau\wedge t} \gamma'(u)W^\Q_u du \right)\Big \vert
 B\Big]\label{ratio},
\end{align}
where $W_t^\Q$ is a $\Q$-Brownian motion starting at $W_s^\Q=Z(s)=z$.
\end{theorem}
\begin{proof}
The proof is based on the Girsanov theorem, and It\^o's formula. Consider
the supermartingale $M=(M(t) \colon {t\geq s})$ defined by
$$
\mathcal{M}(t)=\exp\left(-\int_s^{t\wedge\bar \tau} \gamma(u) dB(u) -\frac
12\int_s^{t\wedge\bar \tau} \gamma(u)^2 du  \right)
$$
Novikov's condition guarantees that $M$ is a martingale. By the 
Girsanov
theorem,
 $$\frac{d\Q}{d\P} \Big |_{\mathcal{F}_{\bar \tau\wedge t}}=\mathcal{M}(t),$$
 and under $\Q$ the process $W^\Q_{t\wedge\bar \tau}=z+B(t\wedge\bar
 \tau)-B(s)+\int_s^{t\wedge\bar \tau}\gamma(s)ds=Z_{t\wedge \bar \tau}$
 is a standard Brownian motion starting at $W_s=z$ and 
stopping at time
 $t\wedge\bar \tau$. By It\^o's formula and the differentiability assumption
 of $\gamma(t)$, we have
\[\gamma({t\wedge\bar \tau})W^\Q_{t\wedge\bar \tau}=\int_s^{t\wedge\bar \tau}
\gamma(u)dW^\Q_u+\int_s^{t\wedge\bar \tau} \gamma'(u)W^\Q_u du+\gamma(s)W_s.
\]
Thus,
\begin{align*}
\frac{1}{\mathcal{M}_{t\wedge\bar \tau}}&=\exp\Big(\int_s^{t\wedge\bar \tau}
\gamma(u) dB(u) +\frac 12\int_s^{t\wedge\bar \tau}\gamma(u)^2 du  \Big)\\
&=\exp\Big(\int_s^{t\wedge\bar \tau} \gamma(u) dW^\Q_u -\frac
12\int_s^{t\wedge\bar \tau} \gamma(u)^2 du  \Big)\\
&=\exp\left( \gamma({t\wedge\bar \tau})W^\Q_{t\wedge\bar
\tau}-\int_s^{{t\wedge\bar \tau}} \gamma'(u)W^\Q_udu-\frac
12\int_s^{{t\wedge\bar \tau}} \gamma(u)^2 du  \right).
\end{align*}
Therefore,
\[\frac{d\P}{d\Q}=\frac{1}{\mathcal{M}_{t\wedge\bar \tau}}=\exp\Big(
\gamma({t\wedge\bar \tau})W^\Q_{t\wedge\bar \tau}-\int_s^{{t\wedge\bar \tau}}
\gamma'(u)W^\Q_udu-\frac 12\int_s^{{t\wedge\bar \tau}} \gamma(u)^2 du  \Big)\]
 and $$\P(B)=\EE^\Q\left[\frac{1}{\mathcal{M}_{t\wedge\bar \tau}} I(\omega\in
 B)\right]=\EE^\Q\left[\frac{1}{\mathcal{M}_{t\wedge\bar\tau}}\Big \vert
 B\right]\Q( B).$$
\end{proof}

To generate an exact sample path of $Z$, we follow the \textit{localization}
method developed by \citet{chen2012} and \citet{giesecke2011}. Suppose we
have generated an exact sample of $(Z_u \ : \ u\leq s)$ for some random or
fixed time  $s< T_f \wedge \tau_{u,v}$. Define
$$\tau_a= \inf\{t\geq 0: |Z(t+s)-Z(s)|\geq a\},$$ where  $a=\min\{|u-y|,\
|v-y|,\  \theta\}$ for some $\theta>0$. By following an 
acceptance/rejection
procedure,  we generate a sample path of
$$\Big (Z(t):s\leq u\leq s+(\tau_a\wedge \Delta)\Big )$$ for some parameter
$\Delta$.

First, consider a sample path $\omega$ of $( W_u^\Q: s\leq u\leq s+\tau_a
\wedge \Delta )$, where $$\tau_a=\inf\{t>0: |W_{t+s}^\Q-W_{s}^\Q|\geq a\}.$$
For ease of exposition, we denote
$$W_u=W_{u+s}^\Q-W_{s}^\Q$$
for $0\leq u\leq \tau_a \wedge \Delta$, which is a standard Brownian motion
under measure $\Q$. Now, given a sample path $\omega$ as a candidate,
we construct a Bernoulli random variable $I$ with success probability
proportional to
\[
\P(I=1|\omega)\propto \frac{d\P(\omega)}{d\Q(\omega)}= \exp\Big(
\gamma(\tau+s)W_\tau-\int_0^{\tau} \gamma'(s+u)W_udu-\frac 12\int_0^{\tau}
\gamma(u+s)^2 du  \Big),
\]
where $\tau=\tau_a\wedge \Delta$. The candidate path $\omega$ is accepted
as a sample path of $\Big(Z(t):s\leq t\leq s+(\Delta\wedge\tau_a )\Big)$
if $I=1$; otherwise, we repeat the procedure.

 The Bernoulli indicator $I$ can be constructed by sampling the jump times
 of a 
doubly-stochastic  Possion process.  The continuously differentiable
 assumption implies that $\gamma(\cdot)$ and $\gamma' (\cdot)$ are bounded
 over the time interval $[s,s+\Delta]$. Let $m=\max_{t\in [s,s+\Delta]}
 |\gamma'(t)|$, and $\tilde m=\max_{t\in [s,s+\Delta]} |\gamma(t)|$.  Let
$$0\leq \phi_t=\gamma'(s+t)W_{t}+m a\leq 2 m a,$$ for every $0\leq t \leq
\Delta$. Let $V$ be a doubly-stochastic  Poisson process with intensity
$\phi_t$  for $0\leq t \leq \Delta$. The required indicator can be generated
by sampling the jump times of $V$ in $[0,\tau]$. The conditional probability
that no
jump occurs in the interval $[0,\tau]$ of the doubly-stochastic Poisson
process $V$ is $\exp\left(\int_0^{\tau} \phi_t dt\right)$.

Let $E_1$ and $E_2$ be two independent exponential random variables with
intensities $$\lambda_1=\tilde m a-\gamma(\tau+s) W_{\tau}\geq 0$$ 
and
$$\lambda_2= -\frac{1}{2} \int_s^{\tau+s} \gamma^2(u) du+ma(\Delta-\tau)\geq
0.$$  The success probability of Bernoulli random $I$ can be defined by
\begin{align*}
\P(I=1|\omega)= \exp(- \int_s^{\tau+s}\phi_u du)\times \P(E_1>\tau)\times
\P(E_2>\tau).
\end{align*}
Observe 
that
\begin{align}
\frac{d\P(\omega)}{d\Q(\omega)}={}&\exp(- \int_0^{\tau} (\gamma'(u+s)W_{u}
du+ma) du \Big)\nonumber\\
&\times \exp\Big(\gamma(\tau+s) W_{\tau}-\tilde m a)\nonumber\\
&\times \exp\Big(-\frac{1}{2} \int_s^{\tau+s} \gamma^2(u) du+ma(\tau-\Delta)
\Big) \exp(ma\Delta+\tilde m a )\nonumber\\
={}& \exp(- \int_0^{\tau}\phi_udu)\times \exp\Big(\gamma(\tau+s) W_{\tau}-\tilde
m a)\nonumber\\
&\times \exp\Big(-\frac{1}{2} \int_s^{\tau+s} \gamma^2(u) du+ma(\tau-\Delta)
\Big) \exp(ma\Delta+ \tilde m a)\nonumber\\
={}& \P(I=1|\omega)\times \exp(ma\Delta+ \tilde m a).\label{bound_ratio}
\end{align}
Generating 
exponential random variables $E_1$ and $E_2$ is
straightforward. The intensity $\phi_t$ is bounded above by $2 m  a$,
allowing us to simulate the event times of $V$
by thinning a Poisson process with intensity $2 m  a$. These properties
facilitate generating the Bernoulli
indicator for the acceptance test of a proposal skeleton.

 More precisely, let $\kappa_1,\kappa_2,\ldots,\kappa_b<\tau$ be the jump
 times of a Poisson process with rate $2 m a$, and $W_{\kappa_i}$ be the
 corresponding values of the 
candidate sample path for $i=1,\ldots,b$. Let
 $(U_1,U_2,\ldots,U_{b+2})$ be a sequence of $b+2$ uniform random
 variables. We accept the candidate path 
if
\begin {align}\label{acceptance_condition}
2ma U_i&> \gamma'(\kappa_i+s)W_{\kappa_i}+ma \mbox {   \ \ \ \ \ \ \ \ \
for } 1\leq i\leq b\\
 U_{b+1}&<\exp \Big( \gamma(\tau+s) W_{\tau}-\tilde m a\Big)\nonumber\\
U_{b+2}&<\exp \Big(\frac{-1}{2} \int_s^{\tau+s} \gamma^2(t) dt+a m
(\tau-\Delta) \Big).\nonumber
\end{align}

The 
sampling of an exit time $\tau_a$ for a Brownian motion $W^\Q$
is possible by following the method in \cite{Burq2008}. Moreover,  in
\cite{chen2012}, it is shown that exact sampling of $W$ at a sequence of
instances before $\tau$ is possible. Given a skeleton of the process at points
$\kappa_1,\kappa_2,\ldots,\kappa_b,\tau$, we can sample the maximum of the
process over $[0,\tau]$. For every $i=1,\ldots,b$, observe that $\widetilde
W_t= W_t+a$ is a Brownian meander for  $\kappa_i\leq t\leq \kappa_{i+1}$
given that $t\leq \tau_a$. Maxmeander algorithm in \cite{Devroy2010}
generates exact samples of the maximum of the Brownian meander and the time
at which the maximum occurs in constant expected time.
\subsection{Summary of the Procedure}
 For the reader's convenience, we summarize our basic algorithm for
 generating exact samples of the maximum of time-dependent BM over a
 finite time interval $[0,\tilde \tau] $, where $\tilde \tau =T_f\wedge
 \tau_{u,v}$. Let $\zeta=\displaystyle\sup_{0\leq t\leq \tilde \tau} Z(t)$
 and $t_\zeta$ be the time at which the maximum occurs. The algorithm
 generates the triplet $(t_\zeta,\zeta, Z_{ \tilde \tau})$. This procedure
 can be used as a subroutine in Step 2 of Algorithm  \ALGExact.

 The initial conditions are 
$\zeta=0$, $t_\zeta=0$, $n=1$, $s_1=0$, and
 select $\theta >0$ and $\Delta<T_f$ (see Remark \ref{para_selection}).
  \begin{description}
\item[Subroutine 1: Exact Sampling of $(t_\zeta,\displaystyle\sup_{t\leq
\tilde \tau}Z(t), Z_{ \tilde \tau})$]\label{ALG3}
\hspace{6cm}
\begin{enumerate}[1]
\item Set 
$a=\min\{|u-Z(s_n)|,|v-Z(s_n)|,\theta\}$.
\item Sample $\tau_a=\inf\{t\geq0  \ : \ |W_t|\geq a\}$; see \cite{Burq2008}.
\item Sample jump times $\kappa_1,\kappa_2,\ldots,\kappa_b<\tau_a\wedge
\min\{T_f-s_n,\Delta\}$ of a Poisson process with rate $2 m a$. Set
$\kappa_{b+1}=\tau_a\wedge \min\{T_f-s_n,\Delta\}$.
\item Sample $W_{\kappa_1},W_{\kappa_2},\ldots,W_{\kappa_b}, W_{\kappa_{b+1}},
W_{\tau_a}$ conditional on $\tau_a$; see p. 11 
of \cite{chen2012}.
\item Accept/reject the proposal skeleton
$$W_{\kappa_1},W_{\kappa_2},\ldots,W_{\kappa_b}, W_{\kappa_{b+1},}
W_{\tau_a}$$ as a sample of the skeleton
$(Z_{s_n+\kappa_1},Z_{s_n+\kappa_2},\ldots,Z_{s_n+\kappa_b},
Z_{s_n+\kappa_{b+1}}  )$ if condition (\ref{acceptance_condition}) holds.
\begin{description}
\item[-] If the proposal is rejected, go to Step 2.
\item[-] If the proposal is accepted, set $s_{n+1}=s_n+\tau_a\wedge
\min\{T-s_n,\Delta \}$, and continue.
\end{description}
\item For $i=1,\ldots,b$, sample jointly $\mu_i=\sup_{\kappa_i\leq t\leq
\kappa_{i+1} } W_t$ and $t_{\mu_i}$, 
the location of the maximum over $[ \kappa_i,
\kappa_{i+1}]$ conditioned on $\tau_a$ as the maximum of a Brownian meander;
see maxmeander algorithm in \citet{Devroy2010}.\\
 If  $\zeta< \mu_i$, update $\zeta\leftarrow \mu_i$ and $t_\zeta\leftarrow
 t_{\mu_i}$.
 \item	 If  $s_{n+1}<T$, increase $n\leftarrow n+1$, and go to Step
 1. Otherwise, stop and return $(t_\zeta,\zeta,Z_{s_{n+1}})$.
\end{enumerate}
\end{description}
\begin{figure}[ht]\label{algorithm3_figure}
 \caption{\small{Sample path of $Z(t)$ over $[s_n,s_{n+1}]$ generated by
 Subroutine \ALGAR. }}
\centering
\includegraphics[ width=120 mm,height=70mm]{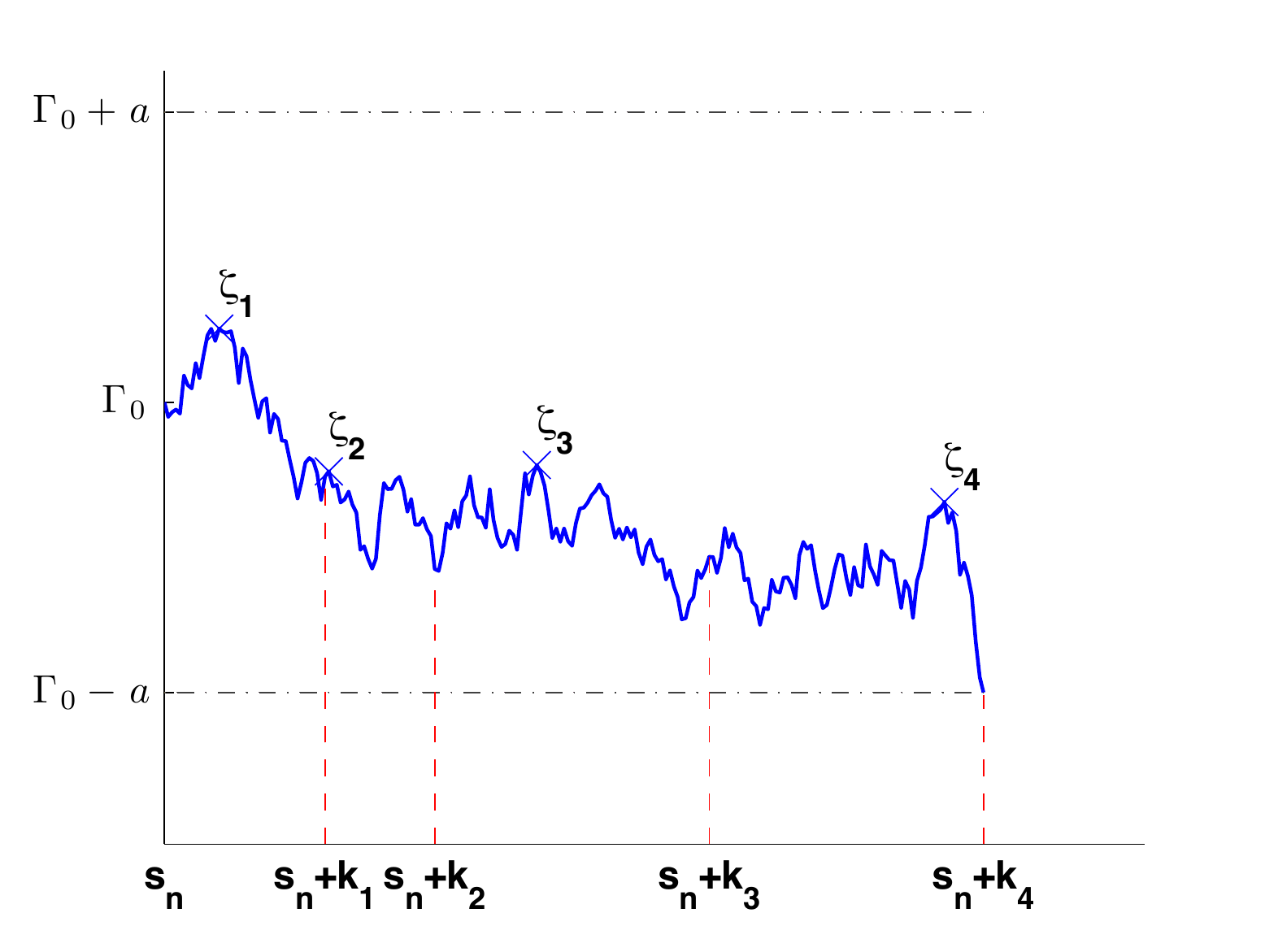}
 \end{figure}

The running time of this algorithm can be bound by $\mathcal{O}(1/|\bar
\gamma|)$. The following theorem shows this result.
\begin{assumption}\label{selecta}Assume that $$\frac{\theta^2}{\Delta}-
\theta (\tilde m+\Delta m/2)\geq 2\log 2,$$ where 
$m=\sup_{u\geq 0}\gamma'(u)$
and  $\tilde m=\sup_{u\geq 0}\gamma(u)$.
\end{assumption}
\begin{theorem}\label{RT-Alg2} Assume \ref{A1},  \ref{A2}, 
and \ref{selecta}
hold. For every $y>0$, let
$$\tau_{-y}=\inf\left\{ t \geq 0 \colon \ Z(t)\leq -y\right\}.$$
The running time of generating a skeleton of the sample path $\left (Z(t)
\ :  0\leq t\leq \tau_{-y}\right)$ by using the Subroutine \ALGAR\ can be
bounded by
 $\mathcal{O}\left(\frac{y+d+\theta}{|\bar \gamma|}\right)$ in expectation.
\end{theorem}
\begin{proof}
Let $N$ be the number of time intervals $[s_n,s_{n+1}]$  before the
termination of the procedure. Also, assume that $\tau_n=s_{n+1}-s_n$
for $n=1,\ldots,N$. According to Lemma \ref{running_time_sample},  we can
show that $\EE[R]$, the expectation of the running time of generating the
skeleton of the sample path $\left (Z(t) \ :  0\leq t\leq \tau_{-y}\right)$,
is 
bounded by
\begin{align}
\EE[R]\leq cm\theta e^{m \theta\Delta+\tilde m \theta} (1+\Delta) \EE
[N]\label{bound_runningtime}
\end{align}
for some constant $c$. By employing Lemma  \ref{tau_lower}, we can bound
$\EE[N]$. We have
\begin{align*}
\EE^{\P}[\tau_{-(y+\theta)}]&\geq \EE\left[ \sum_{n=1}^N \tau_n\right]\\
&\geq\EE\left[ \sum_{n=1}^\infty I(n\leq N) \EE^\P[\tau_n|Z(s_n)]\right]\\
&\geq\EE[N]\frac{ \Delta}{4}.
\end{align*}
The last inequality follows from Lemma \ref{tau_lower}. Therefore,
$$
\EE[N] \leq\frac{4}{ \Delta}\EE^{\P}[\tau_{-(y+\theta)}].
$$
Let $\tilde \tau_{y+\theta}=\inf\{t\geq 0 \colon B(t)-t\cdot\bar\gamma\leq
y+\theta+d\}$ be the first hitting 
time of the dominating constant drift Brownian motion.
It is clear that $\tau_{y+\theta}\leq \tilde \tau_{y+\theta+d}$ by Assumption
(\ref{A2}). From optional sampling theorem (\citet[p. 19]{Shreve}), we can
conclude that
\begin{align*}
       \EE^{\P}[\tau_{-(y+\theta)}]\leq \EE^{\P}[\tilde \tau_{y+\theta+d}]\\
       \leq  \frac{y+\theta+d}{|\bar \gamma|}.
\end{align*}
Therefore, we obtain
\begin{align}
\EE[N]\leq \frac{4(y+\theta+d)}{ \Delta|\bar \gamma|}\label{bound_N}.
\end{align}
From 
inequalities (\ref{bound_runningtime}) and (\ref{bound_N}), we
conclude that
$$\EE[R]\leq  \frac{4(y+\theta+d)}{ \Delta|\bar \gamma|} cm\theta e^{m
\theta \Delta+\tilde m \theta} (1+\Delta).
$$
Thus, the expected running time $\EE
R=\mathcal{O}\left(\frac{y+\theta+d}{|\bar\gamma|}\right)$.
\end{proof}
\begin{remark}\label{para_selection}In order to minimize the upper bound
of the running time expectation, we can choose $\Delta=\frac{1}{m \theta}$
and select $\theta$ such that assumption (\ref{selecta}) holds. A trial and
error procedure might also help to choose the optimal $\Delta$ and $\theta$.
\end{remark}
\begin{lemma}\label{running_time_sample}Assume that (\ref{selecta})
holds. The expectation of the running time for generating the skeleton of
$\left (Z(t) \ :  s_n\leq t\leq s_{n+1}\right)$ is $$\mathcal{O}(mae^{m
a\Delta+\tilde m a} (1+ \Delta)).$$ \end{lemma}
\begin{proof}
The expectation of $b$, the number of events of the Poisson process with
rate $2ma$ that occur in the time interval $[0,\tau_a\wedge \Delta]$ is
$$2ma\EE^Q[\tau_a\wedge \Delta].$$
Sampling $W_{k_i}$ from a Brownian meander for each $i=1,\ldots,b+1$ in Step
3 is possible in a constant time. Therefore, the running time of generating
a skeleton for each candidate sample path is
$$cma(1+\EE^Q[\tau_a\wedge \Delta])$$
for some constant $c$. The probability of accepting each candidate sample
path is
$$\P(I=1)=\exp(-ma\Delta-\tilde m).$$
 Therefore, the expected number of candidate sample paths 
which are
 generated 
before acceptance occurs 
is $\exp(ma\Delta+\tilde m)$. Thus,
 the expected running time to generate a skeleton for $\left (Z(t) \ :
 s_n\leq t\leq s_{n+1}\right)$ is at most
  $$c mae^{m a\Delta+\tilde m a} (1+\EE^Q[\tau_a\wedge \Delta])\leq c mae^{m
  a\Delta+\tilde m a} (1+ \Delta)$$
 for some constant $c$.
 \end{proof}
 \begin{lemma}\label{tau_lower}
 Under assumption (\ref{selecta}), we have
 $$\EE^\P [\tau_n\mid Z(s_n)]\geq  \Delta/4,$$
 where 
$\tau_n=\inf\{t\geq 0\colon |Z(t+s_n)-Z(s_n)|\geq \theta \}$.
 \end{lemma}
 \begin{proof}
By applying Theorem \ref{likelihood} and equality (\ref{bound_ratio}), we have
\begin{align*}
\P\left(\sup_{u\leq \Delta/2 }|Z({u+s_n})-Z(s_n)|\geq \theta \right)&=
\EE^Q\left [ I\left(\sup_{u\leq \Delta/2} |W_u|\geq \theta\right)\frac{\P(W\in
d\omega)}{\Q(W\in d\omega)} \right]\\
&\leq \exp(m\theta \Delta/2+\tilde m \theta) \Q\left(\sup_{u\leq \Delta/2}
|W_u|\geq \theta\right)\\
&\leq  2\exp(m\theta \Delta/2+\tilde m \theta-\theta^2/\Delta).
\end{align*}
The 
last inequality follows from 
Doob's martingale inequality. From
Assumption (\ref{selecta}), we can conclude that
$$
\P(\tau_n\geq  \Delta/2)= \P\left(\sup_{u\leq  \Delta/2
}|Z_{u+s_n}-Z(s_n)|\leq \theta \right)\geq 1/2.
$$
Therefore, we 
have
$$\EE^\P [\tau_n|Z_{s_n}]\geq  \frac{\Delta}{2} \P( \tau_n\geq
\Delta/2)\geq\frac{\Delta}{4}.$$
\end{proof}

\section[Exact Sampling of Time-dependent Drift Brownian Bridge]{Exact
Sampling of Time-dependent Drift Brownian Bridge}\label{SBB}
This section provides an exact method for sampling the maximum of a
unit-volatility time-dependent drift Brownian bridge and  the time at which
this maximum occurs. We define a time-dependent drift Brownian bridge as a
time-dependent Brownian motion given the prescribed values at the beginning
and end of the process. Let $Z^{Br,r}=(Z(t) \colon 0\leq t\leq r)$ be a
time-dependent drift Brownian bridge given $Z_r=y$. Recall that
$$Z(t)=B(t)+\int_0^t\gamma(u)du.$$
 The objective of this section is to generate an exact sample of the
 maximum of
 $$m(r)=\sup_{0\leq s\leq r} Z(s)$$ conditioned on $Z(r)=y$ and the time
 $\eta_r$ at which the maximum occurs.	As we proposed in Algorithm 1, the
 procedure of sampling the joint 
variables $(\eta_r,m(r))$ conditioned
 on $Z(r)=y$ can be employed as a subroutine to generate exact samples
 of ($\eta_\infty,\sup_{t\geq 0} Z(t))$ as an alternative to Algorithm
 \ALGExact. Although, as we discussed earlier, Algorithm \ALGExact\ is more
 efficient 
compared to this approach.

The procedure of generating exact samples of $(\eta_r,m(r))$ conditioned on
$Z(r)=y$ is similar to Subroutine 1, the exact sampling of the time-dependent
Brownian motion. The main difference is that it uses a Brownian bridge rather
than a standard Brownian motion as a candidate sample path. We accept this
candidate as a sample path of $Z^{Br,r}$ with the probability proportional
to the likelihood ratio between the law of two processes. Theorem \ref{LR_BR}
computes 
this likelihood ratio.

Similar to Subroutine \ALGAR\, 
and following the localization method, the
sample path of $Z^{Br,r}$ can be generated piece by piece over the time
intervals $[s_n,s_n+\Delta\wedge \tau_a]$ where
$$\tau_a=\inf\{t\geq 0\colon |Z(t+s_n)-Z(s_n)|\geq a\}$$
for appropriate parameters $a$ and $\Delta$ and given $(s_n,Z(s_n))$.

Assuming $s_n+\Delta< r$, sampling $Z({s_n+\Delta})$ given $Z(r)=y$ is
straightforward. Let
$$\tilde x=x-\int_0^{s_n+\Delta}  \gamma(u)du, \tilde y=y-\int_0^{r}
\gamma(u)du.$$ Observe that
\begin{align*}
\P\Big (Z({s_n+\Delta})\in dx \mid Z_{r}=y,  Z(s_n)\Big)= \P\Big
(B({s_n+\Delta})\in d\tilde x \mid B({r})=\tilde y,  B({s_n})\Big),
\end{align*}
which is the distribution of a standard Brownian bridge. Sampling
of Brownian bridge is 
well known, and one can conveniently generate a
sample of $Z({\Delta}+s_n)$ given $Z_{r}=y$ and $Z(s_n)$ . Therefore, it
suffices to generate a sample of the process $Z^{Br,r}$ (and its maximum)
over the time interval $[s_n,s_n+\Delta\wedge \tau_a]$ given $Z(s_n)$ and
$Z({s_n+\Delta})$. By using 
 time shifting, we can assume that $s_n=0$. Thus,
the problem reduces to generating 
an exact sample of $(\eta_\Delta,\sup_{0\leq
t\leq \Delta}Z(t))$ given that $Z( \Delta )=x$.

The next theorem provides the likelihood ratio between the law of the
process $Z^{Br,\Delta}$ and a standard Brownian bridge, which is used as
the acceptance probability in the 
acceptance/rejection scheme.
\begin{theorem}\label{LR_BR}
Assume that $\gamma(s)$ is a continuously differentiable function, and let
$\tau_a=\inf\{t\geq 0\colon |Z(t)-Z(0)|\geq a\}$. Then, for some constant
$c_1$, we 
obtain
\begin{align}
 \frac{\P\Big(\ (Z(t))_{0\leq t\leq \Delta\wedge \tau_a}\in d\omega\Big\vert
 Z(\Delta)=x\ \Big)}{\Q\Big((W_t)_{0\leq t\leq \Delta\wedge \tau_a}\in
 d\omega\Big\vert W_\Delta=x\Big)}=c_1 \P(I=1\vert\omega)\times
 e^{\psi(\omega)},
  \end{align}
  where
\[
\psi(\omega)=\left\{
  \begin{array}{l l}
   0 & \quad \text{if $\tau_a\geq \Delta$}\\
\frac{1}{2(\Delta-\tau_a)}\left((x-\tilde a)^2-(x-\tilde
a-\int_{\tau_a}^\Delta \gamma(u)du)^2\right)  & \quad \text{if $\tau_a<
\Delta$,}
  \end{array} \right.\]
$\tilde a=Z({\tau_a})=\pm a$,
and $\P(I=1\vert\omega)$ is defined in
  (\ref{bound_ratio}).
 \end{theorem}
  The proof is in the 
Appendix. Observe that $$\psi(\omega)\leq \tilde
  m/2(x+a+\tilde m \Delta),$$ where $\tilde m=\sup_{0\leq t\leq
  \Delta}\gamma(t)$.
Therefore, it is easily conceivable to construct a Bernoulli random variable
with success probability proportional 
to $\exp(\psi(\omega))$.

Constructing the indicator $I$ is possible by modifying Steps 2, 3,
and 
4 in Subroutine \ALGAR. The main difference is in sampling $(\tau_a,
W_{\kappa_1},W_{\kappa_2},\ldots,W_{\kappa_b})$ conditional on $W_\Delta=x$,
where $\kappa_1,\kappa_2,\ldots,\kappa_{b}<\tau_a\wedge \Delta$ are the
jump times of a Poisson process with rate $2ma$. The procedure consists of:
\begin{enumerate}[\hspace{6 mm} $1^\prime$]
\setcounter{enumi}{1}
\item Sample $(\tau_a\wedge \Delta,W_{\tau_a\wedge \Delta})$ given that
$W_\Delta=x$.
\item Sample $\kappa_1,\kappa_2,\ldots,\kappa_b<\tau_a\wedge \Delta$, 
which
are the jump times of a Poisson process with rate $2 m a$.
\item Consider two different scenarios:
\begin{description}
\item[-]  If  $\tau_a\leq  \Delta$,  sample
$W_{\kappa_1},W_{\kappa_2},\ldots,W_{\kappa_b}$ conditional on
$(\tau_a,W_{\tau_a})$.
\item[-]  If  $\tau_a\leq  \Delta$,  sample
$W_{\kappa_1},W_{\kappa_2},\ldots,W_{\kappa_b}$ conditional on $W_\Delta=x$
and $\tau_a\leq  \Delta$.
\end{description}
\end{enumerate}
The other steps of the procedure to generate a sample of
$(\eta_\Delta,m(\Delta))$ are exactly same as steps 5 and 
6 in Subroutine \ALGAR.

Step 
$2^\prime$ is viable by using Theorem \ref{tau_Bridge} in which
we compute the likelihood ratio between distributions of the first hitting
time of a Brownian bridge and a standard Brownian motion.
\begin{theorem}\label{tau_Bridge}  Let $\tau_a= \inf\{t\geq 0: |W_{t}|\geq
a\}$ be the first hitting time of a standard Brownian motion $W$. Suppose
that $r=\frac{|x|}{a}<1$, then we have
\begin{align}
\Q(\tau_a\geq\Delta \vert
W_{\Delta}=x)=\left(1-e^{-\frac{2(1-r)}{\Delta}}\right)\cdot p\left
(0,1,\Delta,1-r\right),
\end{align}
where 
$p(s,x,t,y)$ is defined in (\ref{Ap2}).
Moreover, for some constant $c_2$, we have
\begin{align}
\frac{\Q\left(\tau_a\in dt, W_ {\tau_a}=\tilde a\vert
W_{\Delta}=x\right)}{\Q\left(\tau_a\in dt, W_ {\tau_a}=\tilde
a\right)}=c_2g\left(\frac{x-\tilde a}{\sqrt{\Delta-t}}\right)\leq
\frac{c_2}{\sqrt{2\pi e}\cdot |a-|x||}\label{AR_tau}
\end{align}
 for every $t< \Delta$, where $g(\cdot)$ is the Gaussian function and
 $\tilde a=\pm a$.
\end{theorem}
The definition of the function $p(s,x,t,y)$ and the proof of the theorem
can be found in the 
Appendix.
This theorem facilitates 
step $2^\prime$ in the above procedure. If $|x|>a$,
it is clear that $\tau_a<\Delta$. Suppose that $|x|<a$. 
Sampling the
indicator $I(\tau_a\geq\Delta)$ is possible by generating two independent
uniform random variables $U$ and $V$. If
\[
U< p\left(0,1,\Delta,1-r\right)\mbox{ and\
}V<\left(1-e^{-\frac{2(1-r)}{\Delta}}\right)<1
\]
hold true, we set $\tau_a\geq\Delta$, 
determining whether
$U<p\left(0,1,\Delta,1-2r\right)$ is possible in finite time by Proposition
4.1 in	\citet{chen2012}.

The second part of Theorem \ref{tau_Bridge} assists in generating 
a sample
of $(\tau_a,W_{\tau_a})$ conditional on the events 
$\tau_a<\Delta$ and
$W_\Delta=x$. We can use the acceptance/rejection method. One can generate
a sample of
$(\tau_a,W_{\tau_a})$ for a standard Brownian motion by following the proposed
method in \cite{Burq2008}. According to equation (\ref{AR_tau}), this sample
might be accepted with the probability proportional to $g\left(\frac{x-\tilde
a}{\sqrt{\Delta-t}}\right)$ as a sample of $(\tau_a,W_{\tau_a})$ given that
$\tau_a<\Delta$ and $W_\Delta=x$.

Step 2 of the above procedure is clear. Furthermore, if $\tau_a\leq \Delta$,
generating $ W_{\kappa_1},W_{\kappa_2},\ldots,W_{\kappa_b}, W_{\tau_a}$
conditional 
on $\tau_a$ is possible by procedure (17) in \cite{chen2012}. Now,
we can assign  $I=1$ if condition (\ref{acceptance_condition}) holds.

 In the rest of this part, we elaborate generating samples of $W$ at a
 sequence of instances $\kappa_1,\kappa_2,\ldots,\kappa_b<\tau_a\wedge \Delta$
 conditioned on $W_{\Delta}=x$ and $\tau_a$ in the case that $\Delta<\tau_a$.

 Let
\[\widetilde W_t =
\begin{cases}
   1-\frac{1}{a} W_{\tau_a-ta^2} & \quad \text{if \ $W_{\tau_a=a}$ }\\
   1+\frac{1}{a} W_{\tau_a-ta^2} & \quad \text{if \ $W_{\tau_a=-a}$ }.
\end{cases}
\]
Thanks to the self-similarity and the time reverse properties of the Brownian
motion, it is easy to verify that $\widetilde W=(\widetilde W_t\colon 0\leq
t\leq \frac{\tau_a}{a^2})$ is a Brownian process given that
$0\leq \widetilde W_{t_1}\leq 2$ and $\widetilde W_{t_{b+1}}= 1-\frac{x}{a}
$. Therefore, we are interested in generating 
samples of $\widetilde W$
at the sequence of instances
$$t_0=\frac{\tau_a}{a^2}\geq t_1\geq \ldots,\geq t_b \geq t_{b+1},$$ where
$t_i=\frac{\tau_a-\kappa_i}{a^2}$ for $i=1,\ldots,b$.

Let $V=(V_t\colon t_{b+1}\leq t\leq t_0)$ be a Brownian meander given
that  $V_{t_{b+1}}=1-\frac{x}{a}$ 
and $V_{t_0}=1$. The exact sampling of
$(V_{t_1},\ldots, V_{t_b})$ is discussed in \citet[Section 6]{Devroy2010}. One
can generate a sample 
of the random vector $(V_{t_1},\ldots, V_{t_b})$ and
accept that as a sample of $(\widetilde W_{t_1},\ldots, \widetilde W_{t_b})$
 with the probability computed in Proposition \ref{Bridge_meander}. The
 acceptance decision can be determined by the method proposed in
 \citet[Section 4.4]{chen2012}.
 \begin{proposition}\label{Bridge_meander}
 For any $t_0\geq t_1\geq \ldots,\geq t_b \geq t_{b+1}$ and $0\leq y_1,\ldots
 y_b\leq 2$, the joint
conditional distribution of $(\widetilde W_{t_1},\ldots, \widetilde W_{t_b})$
has the following likelihood ratio with respect to 
$(V_{t_1},\ldots, V_{t_b})$:
\small{\begin{align*}
\frac{
\Q\left (\widetilde W_{t_1}\in dy_1,\ldots, \widetilde W_{t_b}\in dy_b\Big
\vert  0\leq \widetilde W_{t}\leq 2,  \widetilde W_{t_{b+1}}=
y_{b+1}
,\widetilde W_{t_0}=1\right )
}
{\Q\left (V_{t_1}\in dy_1,\ldots,V_{t_b}\in dy_b\Big \vert  0\leq
V_{t}, V_{t_{b+1}}=y_{b+1},V_{t_0}=1\right)}=
  \tilde c \prod_{i=0}^b
p(t_{i+1},y_{i+1};t_{i},y_i),
\end{align*}}
where 
$y_{b+1}=1-\frac{x}{a}$, $y_0=1$, and $\tilde c>0$ is a constant.
 \end{proposition}
 The proof of this proposition is similar to Theorem 4.2 in \citet{chen2012}.

 Given  
the skeleton $(\widetilde W_{\kappa_1},\ldots, \widetilde
 W_{\kappa_b}, W_\Delta)$, sampling $\sup_{0\leq t\leq \Delta} W_t$
 and the location of maximum time is similar to Step 6 of Subroutine
 \ALGAR. It is sufficient to sample $(\mu_i,t_{\mu_i})$ 
jointly, where
 $\mu_i=\sup_{\kappa_i\leq t\leq \kappa_{i+1} } W_t$, 
and $t_{\mu_i}$ is
 the location of the 
maximum over $[ \kappa_i, \kappa_{i+1}]$ conditional 
on
 $\tau_a$ as the maximum of a Brownian meander; see the maxmeander algorithm
 in \citet{Devroy2010}.
\section{Analysis of the Algorithm \ALGExact} \label{analyseALG}
 In this section, we analyze Algorithm \ALGExact\ and show that the algorithm
 terminates in polynomial time. The running
 time of the algorithm is at most $\mathcal{O}(1/\bar \gamma^2)$ for generating 
an exact sample of $(v_t,M(t))$  which is independent of $t$.
Therefore, the running time of the algorithm for generating an exact sample of triplet $(v_t,M(t),Y(t))$ is
at most $\mathcal{O}(1/\bar \gamma^2)+\mathcal{O}(\log(t))$
in which $\mathcal{O}(\log(t))$ is the running time of reading the input and computing $Y(t)$ conditional on $(v_t,M(t))$ .

The key idea here is constructing a constant drift Brownian motion,
dominating $Z(s)$.

\begin{lemma}\label{dominate}
Let $m_{k}$ be the maximum of the process $Z(t)$ until time $\beta_{k}$,
where $\beta_k$ is such that $Z({\beta_{k}})<\zeta_{k}-cd$. Then,
$Z(t)<m_k$ for all $t<\alpha_{k+1}$. Specifically, if $\alpha_{k+1}=\infty$,
then $m_k$ is the global maximum of $Z(t)$ for all $t\geq 0$.
\end{lemma}
\begin{proof}
By Assumption (\ref{A2}), we have
\begin{align*}
 Z(t)&=Z({\beta_k})+B(t)-B({\beta_k})+\int_{\beta_k}^t \gamma(u)du\\
 &\leq Z({\beta_k})+B(t)-B({\beta_k})+d-(t-\beta_k)\cdot
 \bar\gamma=U_t^{\beta_k}.
 \end{align*}
Observe that $U_{\beta_k}^{\beta_k}=Z({\beta_{k}})+d\leq m_{k}-(c-1)d$.
Recall that
\[
\alpha_{k+1}=\inf \{t\geq \beta_{k} \colon U_t\geq m_{k}\}.
\]
Thus, 
we have $Z(t)<m_k$ for all $t<\alpha_{k+1}$.
\end{proof}
Generating exact samples of $\beta_{k+1}$ and $m_{k+1}$ is possible by
Subroutine \ALGAR.
Note that the process $U^{\beta_k}$ is a Brownian motion with constant
drift. Therefore, we can easily sample the hitting time $\alpha_{k+1}$
based on the following lemma.
\begin{lemma}\label{d_alpha} Let  $x=m_k-U_{\beta_k}^{\beta_k}>0$. Then,
$$\P(\alpha_{k+1}=\infty|U_{\beta_k})=1-\exp(-2\bar\gamma x),$$ and
\begin{align}\label{hit-d}
\P\Big(\alpha_{k+1}\leq t+\beta_k|U_{\beta_k}^{\beta_k},\alpha_{k+1}<\infty
\Big)=\P\Big(IG\left(\frac{x}{|\bar\gamma|},x^2\right)\leq t\Big),
\end{align}
where $IG(\cdot,\cdot)$ denotes the inverse Gaussian-distribution with mean
$\frac{x}{|\bar\gamma|}$ and shape parameter $x^2$.
\end{lemma}
  \begin{remark}
  Draws from the inverse Gaussian distribution can be generated in a very
  efficient way; see the algorithm described
in \citet[Chapter IV]{Devroye-non-uniformrandom}.
\end{remark}
 \begin{proof}By 
Assumption (\ref{A2}), $\bar \gamma>0$, and we have
 from \citet [p. 297]{Shreve} that
\begin{align*}\label{hit-d}
\P\Big(\alpha_{k+1}\geq T+\beta_k|U_{\beta_k}\Big)= \Phi \Big(\frac{x+\bar
\gamma T}{\sqrt{T}}\Big)-\exp(-2\bar \gamma x) \Phi \Big(\frac{-x+\bar
\gamma T} {\sqrt{T}}\Big).
\end{align*}
Therefore, $\P(\alpha_{k+1}=\infty|U_{\beta_k})=1-\exp(-2\bar\gamma x)$, and
\[\P\Big(\alpha_{k+1}\leq T+\beta_k|U_{\beta_k},\alpha_{k+1}<\infty
\Big)=\exp(-2\bar \gamma x)\Phi \Big(\frac{-x+\bar \gamma
T}{\sqrt{T}}\Big)+\Phi \Big(\frac{-x-\bar \gamma T} {\sqrt{T}}\Big),
\]
which is the distribution of the 
inverse Gaussian distribution.
  \end{proof}

 As a result of these two lemmas, we can easily observe that the algorithm
 terminates in a finite number of iterations. The probability that the
 algorithm terminates in iteration $k$ (i.e. $\alpha_{k+1}=\infty$) is at
 least $1-\exp(-2\bar\gamma(c-1)d)$. Moreover, if  $\alpha_{k+1}=\infty$, then
 $\zeta_k$ is the maximum of $Z(t)$ for all $t\geq 0$. So at each step $k$,
 the procedure is terminated with at least constant probability. Therefore,
 the algorithm terminates in finite time almost surely. We
summarize this result in the following theorem.
\begin{theorem}\label{RT_alg1} For every $\bar \gamma>0$, Algorithm
\ALGExact\ terminates in finite time. Furthermore,The expected number of iterations is
at most $(1- \exp{(-2(c-1)d \bar \gamma)})^{-1}$, and the expected running time of generating an exact sample of $(v_t,M(t))$ is
$\mathcal{O}(1/\bar \gamma^2)$ for every $t\in[0,\infty]$.
\end{theorem}
\begin{proof}
It is clear that the running time of generating an exact sample of $(v_t,M(t))$ is bounded for every $t\geq0$
is bounded by the running time of generating an exact sample of $(v_\infty,M(\infty))$. Therefore, we show the result for the recent 
case. Observe that
$$U_{\beta_{k}}^{\beta_k}=Z_{\beta_k}+d\leq m_{k}-(c-1)d.$$ Therefore,
by Lemma \ref{d_alpha}, the probability that the algorithm terminates in
the next step is
\[
\P\Big(\alpha_{k+1}=\infty|U_{\beta_k}\Big)=1-\exp(2
(m_{k}-U_{\beta_{k}})\bar\gamma)
\geq 1-\exp(2(c-1)d\bar\gamma).
\]
Hence, the number of iterations before termination represented by $K$,
is dominated by a geometric random variable. So we have
$$ E [K] \leq  (1- \exp{(-2(c-1)d \bar \gamma)})^{-1}. $$
Given that $\alpha_K=\infty$,  we have
$$Z(t) \leq U_t <m_K$$
for all $t>\beta_K$. Since $m_K$ is the maximum of the process until time
$\beta_k$, we conclude that $\sup_{t\geq 0}Z(t)=m_K. $

 According to Theorem \ref{RT-Alg2}, the expected running time of generating
 a sample of $\displaystyle\zeta_k=\sup_{\alpha_k\leq t<\beta_k} Z(t)$
 by Subroutine \ALGAR\ is $\mathcal{O}(1/|\bar \gamma|)$ for every $1\leq
 k \leq K$. Since,
 $$E[K]=\mathcal{O}(1/|\bar \gamma|),$$
  the expected running time of Algorithm \ALGExact\ is
  $\mathcal{O}\left(1/\bar {\gamma}^2\right)$.
\end{proof}

In the end, it is worth 
mentioning that one can improve the performance
of the algorithm by changing the update rule of $\beta_{k}$:
\begin{align}\label{beta3}
\beta_{k}=\inf \left\{t\geq \epsilon + \alpha_{k} \colon Z(t)\leq m^{
\epsilon}_{k}- cd\right\},
\end{align}
where $ m^{ \epsilon}_{k}=\displaystyle\sup_{t\leq  \epsilon +
\alpha_{k-1}}Z(t)$.
Here, $ \epsilon$ is a parameter to guarantee that each step is not too
short. Our simulation experiments show that choosing reasonable $\epsilon$
improves the running 
time.

\section{Numerical Experiment}
We illustrate the effectiveness and relative performance of the exact sampling
method through the numerical experiment. We apply exact algorithm  \ALGExact\ to
RBM with drift coefficient $$\gamma(u)=\cos(2\pi u)-0.5.$$ In other words,
\begin{align}
dX_t= (\cos(2\pi t)-0.5)dt +dB(t)+ dL_t.\label{cosine}
\end{align}
The drift coefficient $\gamma(u)=\cos(2\pi u)-0.5$ is a periodic function
with period $1$, and $$\bar\gamma=-\int_0^1\cos(2\pi u)-0.5 du=0.5>0.$$
Therefore, Assumption \ref{A2} holds. Recall that
\begin{align}
X(n)\Rightarrow M(\infty)= \sup_{t\geq 0}\Big ( \int_0^t (\cos(2\pi
u)-.5)du+B(t)\Big).\label{TRN}
\end{align}
In this experiment, we compare the discretization method and our exact
algorithm for generating samples of  $M(\infty)$. Conventional discretization
techniques can only approximate samples of $M(\infty)$; the exact algorithm
returns exact samples of $M(\infty)$.
\subsection{Discretization Method}\label{DM}
We compare the exact draws of our algorithms with the approximate ones of the
simple discretization scheme. A naive approach to discretize (\ref{cosine})
is given by
\[X_{t_{i+1}}=\max(X_{t_{i}}+\gamma(t_i)(t_{i+1}-t_i)+B(t_{i+1})-B(t_{i}),0),\]
where $t_i=i\delta$ for $i=1,2,\ldots$ and the step size $\delta>0$. However,
\citet{asmussen1995discretization} shows that this discretization scheme
is highly biased and the bias is at least of order $\delta^{1/2}$. As an
alternative, we employ the discretization scheme for time-dependent RBM
similar to \citet{lepingle1995euler}. In Proposition \ref{opt_step_all},
we show that the bias of this scheme is of order $\delta^{2}$. Here,
we quickly outline the discretization scheme. For time step $\delta>0$,
define the piecewise constant drift Brownian motion
$\hat Z^\delta=\{\hat Z^\delta(t)\}_{t\geq 0}$. Consider the discrete grid
points $t_i=\delta i$ for $i=0,1,2,\ldots$. Let $\hat Z^\delta(t)=0$,
 and for $t_ i\leq t \leq t_{i+1}$
\begin{align}\label{hatX}
\hat Z^\delta (t)\df\hat Z^\delta (t_i)+\gamma_i(t-t_i)+B(t)-B(t_i),
\end{align}
where $\gamma_i=\frac{1}{\delta}\int_{t_i}^{t_{i+1}} \gamma(u)du $.
The joint distribution of $$\Big(\hat Z^\delta (t_{i+1}), \sup_{t_ i\leq
t \leq t_{i+1}}\hat Z^\delta (t)\Big)$$ given $\hat Y^\delta (t_{i})$
is known. Thus, the exact sampling of the process and its maximum over
the grid points is possible. The details of the algorithm are given in
\citet[p. 302]{Glynn-simulation}. By choosing a large $T=t_M$ and sufficiently
small $\delta$, the random variable $\displaystyle\sup_{0\leq t\leq t_M}
\hat Z^\delta (t)$ approximates $M(\infty)$.

In Proposition \ref{opt_step_all}, we discuss how to  allocate the
computational budget of the discretization method between the number of time
step $\delta$ and the number of trials. We show that for the first-order
method of discretization, it is asymptotically optimal to increase the
number of time steps proportional to the fourth root of the number of
replications. However, the optimal constant of proportionality is not known.

\subsection{Results}
We generate samples of $M(\infty)$ defined in (\ref{TRN}) using  exact
Algorithm \ALGExact\ and the discretization scheme. For the  discretization
scheme, we use different increments $\delta$ and fixed time horizon $T=35$.

Table \ref{ks-table} presents the time required to get 
200,000 draws from
$M(\infty)$ for the exact algorithm and  discretization scheme for different
increments $\delta$. Moreover, the $p$-values of the Kolmogorov--Smirnov test
are included, 
which compare the approximate samples of the Euler scheme with
the exact sample for different increments $\delta$. The $p$-values indicate
whether the samples are drawn from the same distribution or not.

Table \ref{unbaised-table} reports the comparison of the 
estimation of
$\EE[M(\infty)]$ by both methods. Motivated by Theorem \ref{opt_step_all},
the step size is set to $\delta=\frac{1}{5}N^{-\frac{1}{4}}$, where $N$ is
the number of simulation trials. The fourth and fifth columns of the table
show the estimation of $\EE[M(\infty]$ and the 
$90\%$ confidence interval for
different numbers 
of trials. The standard error (SE) is estimated as the
sample standard
deviation of the simulation output divided by the square root of the number
of trials.
The bias is given by the difference between the expectation of the estimator
and the true value of $\EE[M(\infty)]$.  Bias of the estimator generated
by the exact method is zero. Thus, the true value is estimated using $2$
million trials generated by the exact method. The bias of the discretization
scheme is estimated by using 800,000 
trails.
The $8^{\mbox{th}}$ column of the table reports the root mean square error
(RMSE) calculated by $\sqrt{\mbox{SE}^2+\mbox{Bias}^2}$. The last column
shows the computational time required to generate different numbers 
of
trials. Simulations were performed on a server with an Intel Core Duo 3.16
GHz processor and 4GB RAM.The code is written in MATLAB Version  (R2011b).

It is remarkable that the exact algorithm is much more efficient than the
discretization method. The bias is zero, the error is less, and even the
algorithm improves the running time.

The convergence rates of the exact and discretization methods are compared
in Figure \ref{CRate}. The exact method achieves an 
optimal convergence rate;
RMSE of the estimator decreases at a rate 
of $\mathcal{O}(1/\sqrt{t})$, where $t$
is the computational budget. The convergence rate of the discretization scheme
is $\mathcal{O}(t^{-2/5})$, 
confirming Theorem \ref{opt_step_all}. Since
the convergence of the 
discretization scheme is slower, we can conclude that
the discretization bias is significant.

\begin{table}[h]
\begin{center}
   \caption{Comparing the samples generated by the discretization scheme and
   exact method by the Kolmogorov--Smirnov test.}
\begin{tabular}{llll  }
\hline
&$\delta$ & p-value & Time (Sec)\\
\hline
\multirow{6}{*}{Discretization}&$2^{-1}$	&			E-152
&	21	\\
&$2^{-2}$	&		E-20	&	29	\\
&$2^{-4}$	&		0.007	&	62	\\
&$2^{-6}$	&			0.002	&	221	\\
&$2^{-10}$	&		0.145	&	2305	\\
&$2^{-12}$	&		0.231	&	11200	\\

\hline
Exact Method & --- & --- &1305\\
\hline
\end{tabular}
\label{ks-table}
\bigskip
\end{center}

Simulation Results of $\EE[M(\infty]$ under Model (\ref{cosine}). The
p-values for the Kolmogorov--Smirnov test with null hypothesis that the
exact and the corresponding approximate draws come from the same distribution.
\end{table}
\begin{figure}[b]
    \caption{\small{Convergence of the RMSEs computed by discretization
    and exact method. }}

\begin{center}
  \includegraphics[ width=100 mm,trim = 0mm 0mm 0mm 5mm,
  clip]{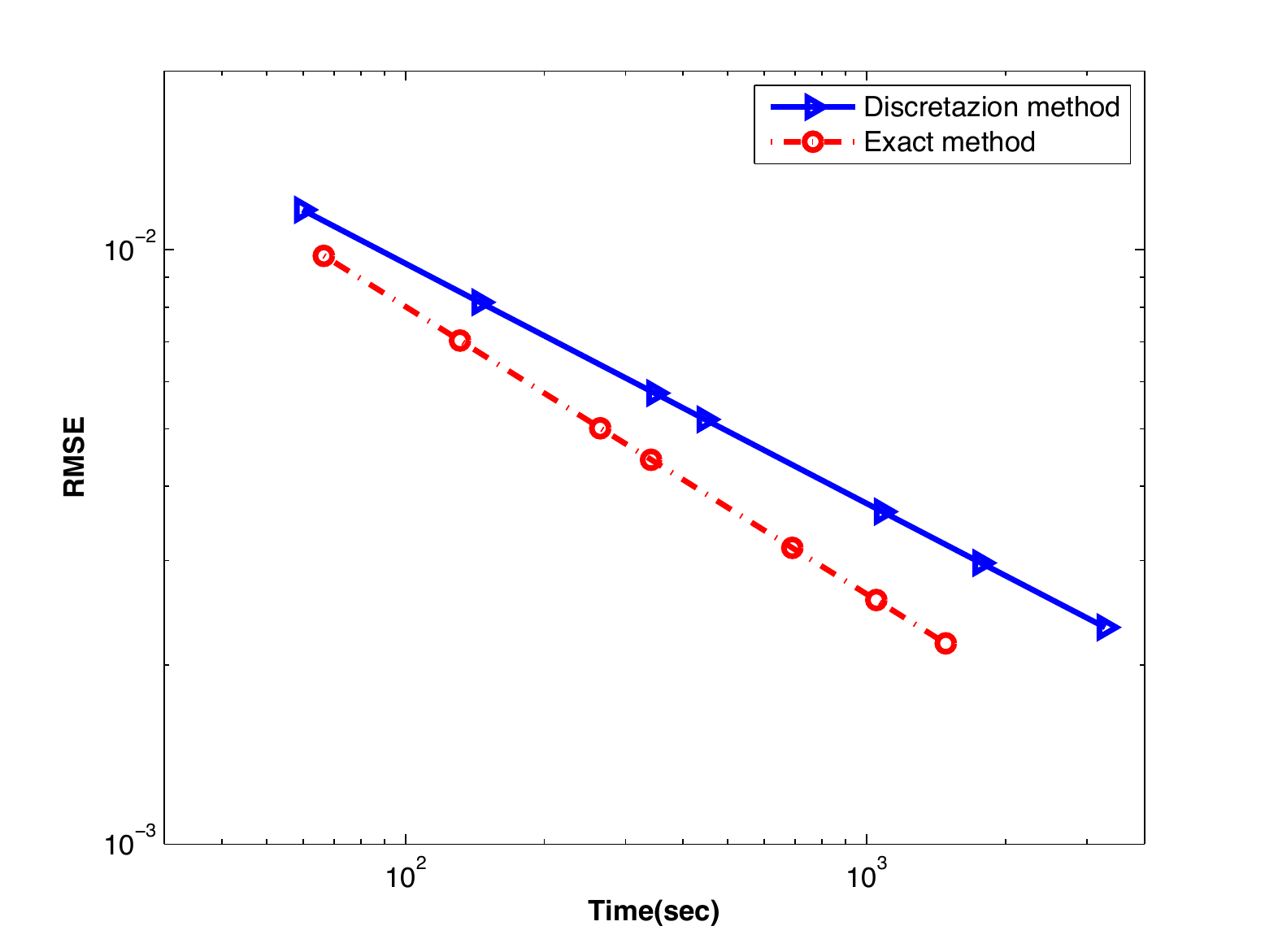}
    \label{CRate}
    \end{center}
    The convergence rate of the exact scheme is $\mathcal{O}(t^{-1/2})$, and the convergence rate of discretization scheme is $\mathcal{O}(t^{-2/5})$.
 \end{figure}

\begin{table}[b]\label{unbaised-table}
  \centering
   \caption{Simulation results of estimation $\EE [M(\infty)]$ under the
   Model (\ref{cosine})}
\scalebox{0.9}{
\begin{tabular}{clcccccccr}
\hline
 & Trials & Steps & Mean			& 90\% CI &
 SE	     & Bias	       &RMSE & Time(sec) \\
 \hline
\multirow{7}{*}{Discretization} & 10000 & 2.00E-03 & 1.0519 & [ 1.0356	,
1.0681	] & 9.89E-03 & 5.68E-03 & 1.14E-02 & 60 \\
 & 20000 & 1.68E-03 & 1.0505 & [  1.0390  ,  1.0620  ] & 7.02E-03 & 4.28E-03
 & 8.22E-03 & 145 \\
 & 40000 & 1.41E-03 & 1.0490 & [  1.0409  ,  1.0571  ] & 4.95E-03 & 2.82E-03
 & 5.70E-03 & 347 \\
 & 50000 & 1.34E-03 & 1.0493 & [  1.0420  ,  1.0565  ] & 4.43E-03 & 3.09E-03
 & 5.40E-03 & 446 \\
 & 100000 & 1.12E-03 & 1.0484 &[ 1.0433  ,  1.0536  ] & 3.12E-03 & 2.24E-03 &
 3.85E-03 & 1081 \\
 & 150000 & 1.02E-03 & 1.0470 & [  1.0428  ,  1.0512  ] & 2.56E-03 &
 8.35E-04 & 2.69E-03 & 1769 \\
 & 200000 & 9.46E-04 & 1.0470 & [  1.0434  ,  1.0507  ] & 2.21E-03 &
 8.58E-04 & 2.37E-03 & 3290 \\

 \hline

\multirow{7}{*}{Exact}
& 10000 &  NA & 1.0477 & [  1.0314 ,  1.0639  ] & 9.91E-03 &0 & 9.91E-03 &
67 \\
 & 20000 & NA & 1.0456 & [  1.0341  ,  1.0571  ] & 7.00E-03 & 0 & 7.00E-03 &
 131 \\
 & 40000 &NA  & 1.0485 & [  1.0404  ,  1.0566  ] & 4.95E-03 & 0 & 4.95E-03 &
 265 \\
 & 50000 & NA & 1.0421 & [  1.0348  ,  1.0494  ] & 4.43E-03 &0 & 4.43E-03 &
 341 \\
 & 100000 & NA & 1.0453 & [  1.0401  ,	1.0504	] & 3.13E-03 & 0 & 3.13E-03 &
 690 \\
 & 150000 & NA & 1.0458 & [  1.0416  ,	1.0500	] & 2.56E-03 & 0 & 2.56E-03 &
 1049 \\
 & 200000 & NA & 1.0468 & [  1.0432  ,	1.0504	] & 2.21E-03 & 0 & 2.21E-03 &
 1484 \\
\hline
\end{tabular}
}
\end{table}
\newpage

\bibliographystyle{apa}
\bibliography{references}

\appendix
\section{Optimal Convergence Rate of Discretization
Scheme}

This appendix presents the optimal tradeoff between choosing the number
of trials $N$ and time step $\delta$ in approximately generating samples
by the discretization scheme discussed in Subsection \ref{DM}. 
Suppose we
want to compute
$\alpha =\EE [M(T)]$, where $$M(T)=\sup_{0\leq t\leq T}
B(t)+\int_0^t\gamma(u)du$$
by Monte Carlo simulation. However, we generate samples of
$$\hat M^\delta=\sup_{0\leq t\leq T} \hat Z^\delta(t)$$
 in which $\hat Z^\delta(t)$ is defined by (\ref{hatX}) using time steps
 of length $\delta$. As an estimator of $\alpha$, we compute
 $$\hat \alpha(\delta,N)=\frac{1}{N}\sum_{j=1}^N \hat M^{\delta}_j, $$
where $\hat M_1^{\delta},\ldots,\hat M_N^{\delta}$ is a sequence of 
i.i.d.
copies of $\hat M^\delta$. The following assumption is required for Theorem
\ref{opt_step_all}.
 \begin{assumption}\label{assumption_allocation}Assume that
 \begin{enumerate}[i.)]
\item $\var [\hat M^\delta] \rightarrow \var [M(T)]$ as $\delta\downarrow 0$.
\item The function $\gamma(u)$ is continuously differentiable.
\end{enumerate}
 \end{assumption}
\begin{theorem}\label{opt_step_all}
Suppose $c$ is the computational budget to approximate $\alpha =\EE M(T)$
by $\hat \alpha(\delta,N)$.
 Assume that \ref{assumption_allocation} holds. Then, it is asymptotically
 optimal to draw
 $N=\mathcal{O}(c^{\frac 4 5})$ trials and choose time step
 $\delta=\mathcal{O}(c^{\frac {-1} {5}})$ to minimize the root mean square
 error (RMSE) which gives RMSE at most $\mathcal{O}(c^{-\frac 2 5})$.
\end{theorem}
 \begin{proof}By using Assumption \ref{assumption_allocation}, we have
 \begin{align}
N \var [\alpha(\delta,N)]&= \frac{1}{N}\sum_{j=1}^N \var [\hat
M^{\delta}_j]\nonumber\\
&= \var [\hat M^\delta] \rightarrow \var[ M(T)]\label{var_lim}
 \end{align}
 as $\delta \downarrow 0$.
 Note that
 \begin{align*}
\EE[ (\hat \alpha(\delta,N)-\alpha)^2]&=  \var [\alpha(\delta,N)]+(\alpha-
\EE [\hat M^\delta])^2\\
  &=  \frac{1}{N}\var [\hat M^\delta]+(\EE [M(T)- \hat M^\delta])^2.
 \end{align*}
The 
first term is bound by $\mathcal{O}(\frac{1}{N})$ asymptotically. The
second term can be bounded over each sample path. Observe that
\begin{align*}
\vert M(T)-\hat M^\delta \vert&= \left\vert \sup_{0\leq t\leq T} Z(t)
-\sup_{0\leq t\leq T}\hat Z^\delta(t)\right\vert\\
&\leq \sup_{0\leq t\leq T} \vert Z(t) -\hat Z^\delta(t)\vert.
\end{align*}
 For every $t_i\leq t<t_{i+1}$ and $i=1,\ldots,{T/\delta}$,  we obtain
 \begin{align*}
\vert Z(t) -\hat Z^\delta(t)\vert &= \left \vert B(t)+\int_{t_i}^t\gamma(u)du
-\left(B(t)+\bar\gamma_i (t-t_i)\right)\right \vert\\
&\leq \left\vert \int_{t_i}^t\gamma(u)du -\bar\gamma_i (t-t_i)\right\vert,\\
&\leq m\delta^2,
\end{align*}
 where $m=\sup_{0\leq t\leq T} \gamma'(t)$ is constant. The last inequality
 is followed by Taylor's theorem and the assumption that $\gamma(u)$ is
 continuously differentiable. Therefore,
 \begin{align}\label{Bias}
\EE[( M(T)-\hat M^\delta)^2 ]\leq m^2\delta^4.
 \end{align}
 By combining  (\ref{var_lim}) and (\ref{Bias}), we have
 \begin{align}
 \EE[ (\hat \alpha(\delta,N)-\alpha)^2]\leq
 m^2\delta^4+\frac{2}{N}\var[M(T)]\label{UBRMSE}
 \end{align}
 for small enough $\delta$. Let $k$ be the constant time required to generate
 an exact sample of
 $$O_i=\sup_{t_i\leq t\leq  t_{i+1}} \hat Z ^\delta(t).$$
  The total computational budget required to draw $N$ independent samples
  of $\hat M^\delta=\max_{1\leq i\leq T/{\delta}} O_i$ is
\[
    c=k\cdot T/ \delta\cdot (N+1).
\]
Therefore, the right hand side of (\ref{UBRMSE}) can be 
minimized by selecting
$N=\mathcal{O}(c^{\frac 4 5})$, and $\delta=\mathcal{O}(c^{-\frac {1}
{5}})$. Moreover, we can conclude that the 
RMSE 
is $\mathcal{O}(c^{-\frac 2 5})$.
\end{proof}
\section{Proof of Theorem \ref{LR_BR}}
\begin{proof}
We consider two different scenarios. First, assume that $\omega \in \{\tau\leq
\Delta\}$. Then, we have
{
 \begin{align*}
 \frac{\P\Big(\ (Z(t))_{0\leq t\leq \Delta\wedge \tau_a}\in d\omega\Big\vert
 Z(\Delta)=x\ \Big)}{\Q\Big((W_t)_{0\leq t\leq \Delta\wedge \tau_a}\in d
 \omega\Big\vert W_\Delta=x\Big)}
 ={}&  \frac{\P\Big(\ (Z(t))_{0\leq t\leq  \tau_a}\in d\omega,Z(\Delta)\in dx\
 \Big)/\P(Z(\Delta)\in dx)}{\Q\Big((W_t)_{0\leq t\leq  \tau_a}\in d \omega,
 W_\Delta\in dx\Big)/ \Q(W_\Delta\in dx)}\\
 ={}& \frac{ \P\Big(\ (Z(t))_{0\leq t\leq  \tau_a}\in d\omega\Big)\cdot
 \P(Z(\Delta)\in dx \vert Z(\tau_a)=\tilde a)} {\Q\Big((W_t)_{0\leq t\leq
 \tau_a}\in d \omega\Big)\cdot	\Q(W_\Delta\in dx\vert W_{\tau_a}=\tilde a)}\\
 &\times	\frac{\Q(W_\Delta\in dx)}{\P(Z(\Delta)\in dx  )}. 
 \end{align*}}
 The 
last equality is obtained by applying the 
strong Markov property of
 Brownian motion.
  Observe that
 \begin{align*}
  \frac{\P(Z(\Delta)\in dx \vert Z(\tau_a)=\tilde a)} {\Q(W_\Delta\in
  dx\vert W_{\tau_a}=\tilde a)}
 & =\exp\left(
		     \frac{1}{2(\Delta-\tau_a)}\left(x-\tilde
		     a-\int_{\tau_a}^\Delta\gamma(u) du\right)^2-(x-a)^2
		       \right)\\
		       &=\exp(\psi(\omega)).
   \end{align*}
 From 
(\ref{bound_ratio}), we have
 \[
 \frac{\P\Big(\ (Z(t))_{0\leq t\leq  \tau_a}\in
 d\omega\Big)}{\Q\Big((W_t)_{0\leq t\leq  \tau_a}\in d
 \omega\Big)}=\exp(ma\Delta+\tilde m a) \cdot\P(I=1\vert \omega).
 \]
 Therefore, letting
 $$c\df\exp(ma\Delta+\tilde m a)\cdot \frac{\Q(W_\Delta\in
 dx)}{\P(Z(\Delta)\in dx  )},$$
  we can conclude the theorem. Similarly, In the case that $\omega \in
  \{\tau\geq  \Delta\}$, we 
have
 \begin{align*}
 \frac{\P\Big(\ (Z(t))_{0\leq t\leq \Delta\wedge \tau_a}\in d\omega\Big\vert
 Z(\Delta)=x\ \Big)}{\Q\Big((W_t)_{0\leq t\leq \Delta\wedge \tau_a}\in d
 \omega\Big\vert W_\Delta=x\Big)}
 &=  \frac{\P\Big(\ (Z(t))_{0\leq t\leq \Delta}\in d\omega
 \Big)/\P(Z(\Delta)\in dx)}{\Q\Big((W_t)_{0\leq t\leq  \Delta}\in d
 \omega\Big)/ \Q(W_\Delta\in dx)}\\
 &=c \cdot \P(I=1\vert \omega).
 \end{align*}
\end{proof}
\section {Proof of Theorem \ref{tau_Bridge}}
\begin{definition}\label{Ap2}
Denote $BB^{x\rightarrow y}_u$ to be the Brownian bridge from $x$ to $y$
on $[s, t]$. Let
\[
     p(s,x;t,y)\df \P\left(0<BB^{x\rightarrow y}_u<2 \mbox{,  for all $u\in
     [s,t]$}\Big \vert	0<BB^{x\rightarrow y}_u  \mbox{,  for all $u\in
     [s,t]$} \right).
\]
In 
\citet{chen2012}, it is shown that
for any $s<t$ and $x,y\in[0,2]$, we have
\begin{align}
 p(s,x;t,y)=\frac{1-\sum_{j=1}^\infty
 (\theta_j-\vartheta_j)}{1-\exp(-2xy/(t-s))},
\end{align}
where
\begin{align*}
&
\theta_j(s,x;t,y)\df\exp\left(-\frac{2(2j-x)(2j-y)}{t-s}\right)+\exp\left(-\frac{2(2(j-1)+x)(2(j-1)+y)}{t-s}\right)\\
&\vartheta_j(s,x;t,y)\df
\exp\left(-\frac{2j(4j+2(x-y))}{t-s}\right)+\exp\left(-\frac{2j(4j-2(x-y)}{t-s}\right).
\end{align*}
\end{definition}
\begin{proofof}{\textit{Theorem \ref{tau_Bridge}}}
 Let $\overline {BB}(t)=1-\frac{1}{a} W_t$ be a Brownian bridge conditional
 on $\overline {BB}(0)=1$ and  $\overline {BB}(\Delta)=1-r$. Then, we have
\begin{align*}
\Q(\tau_a\geq \Delta)&=\Q (-a\leq W_t\leq a\vert W_\Delta=x,W_0=0)\\
&=\Q (0\leq \overline {BB}(t)\leq 2  \mbox{, for all $0\leq t\leq \Delta$})\\
&=p(0,1,\Delta,1-r)\Q\left(  \overline {BB}(t) \geq 0 \mbox{, for all $0\leq
t\leq \Delta$}\right).
\end{align*}
The 
last equality is concluded from p. 23 of \citet{chen2012}. By using the
distribution of the maximum of the 
Brownian bridge, we obtain
\begin{align*}
\Q\left(  \overline {BB}(t) \geq 0 \mbox{, for all $0\leq t\leq
\Delta$}\right)=1-\exp\left(-\frac{2}{\Delta}(1-r)\right).
\end{align*}
The 
second part of the theorem is straightforward.
\end{proofof}

 \end{document}